\documentclass[reqno,10pt]{amsart}

\usepackage{amsmath, amsthm, amsfonts,color}
\usepackage{pdfsync}
\usepackage{parskip}
\usepackage{hyperref}
\usepackage{mathrsfs}  
\usepackage{mathtools}
\usepackage{mdframed}
\usepackage{tikz}
\numberwithin{equation}{section}
\newtheorem{theorem}{Theorem}[section]
\newtheorem{prop}[theorem]{Proposition}
\newtheorem{definition}[theorem]{Definition}
\newtheorem{lem}[theorem]{Lemma}

\theoremstyle{remark}
\newtheorem{remark}[theorem]{Remark}

\def\M{\mathsf{M}}

\def\p{\mathsf{p}}
\def\d{{\sf d}}

\def\sH{\mathscr{H}}
\def\cE{\mathcal{E}}

\def\R{\mathbb{R}}

\def\N{\mathbb{N}}
\def\Nz{\dot{\mathbb{N}}}

\def\bE{\mathbb{E}}

\def\L{\mathcal{L}}
\def\Lis{\mathcal{L}{\rm{is}}}
\def\F{\mathfrak{F}}

\def\O{\mathsf{O}}

\def\Q{\mathsf{Q}^{N}}
\def\Qq{\mathsf{Q}}
\def\Qqq{\mathsf{Q}^2}

\def\cA{\mathcal{A}}
\def\sA{\mathscr{A}}

\def\a{\mathfrak{a}}

\def\Rp{{\rm{Re}}}

\def\div{{\rm div}}

\def\sg{{\rm{sign}}}

\def\RS{\mathcal{RS}}

\def\S{\mathcal{S}}
\def\zW{\prescript{}{0}{W}}

\def\sU{\mathscr{U}_\delta}
\def\sUG{\mathscr{U}_\delta \setminus \Gamma}
\def\bsUG{\bar{\mathscr{U}}_\delta \setminus \Gamma}
\def\sUk{\mathscr{U}_{\delta_k}}

\def\bsUGk{\bar{\mathscr{U}}_{\delta_k} \setminus \Gamma}

\def\TN{\mathbb{T}^N}
\def\GT{\Gamma_{\mathbb{T}}}
\def\QG{\Q\setminus \Gamma}
\def\QqG{\Qq\setminus \Gamma}
\def\QqqG{\Qqq\setminus \Gamma}

\begin{document}

\title[Wellposedness of a nonlocal nonlinear diffusion]{Wellposedness
  of a nonlocal nonlinear diffusion equation of image 
  processing} 

\author{Patrick Guidotti}
\address{University of California, Irvine 
Department of Mathematics 
340 Rowland Hall 
Irvine, CA 92697-3875, USA}
\email{gpatrick@math.uci.edu}

\author[Y. Shao]{Yuanzhen Shao}
\address{Department of Mathematics,
         Purdue University, 
         150 N. University Street, 
         West La\-fayet\-te, IN 47907-2067, USA}
\email{shao92@purdue.edu}


\begin{abstract}
Existence and uniqueness of solutions to non-smooth initial data are
established for a slight modification of the degenerate regularization
of the well-known Perona-Malik equation first proposed in
\cite{GL07}. The results heavily rely on the choice of an appropriate
functional setting inspired by a recent approach to degenerate
parabolic equations via so-called singular Riemannian 
manifolds (\cite{Ama13, Ama13b}).
\end{abstract}
\maketitle

\section{\bf Introduction}
In the early 90ies, P. Perona and J. Malik \cite{PerMal90}  introduced
a novel paradigm by proposing the use of nonlinear diffusions as an
image processing tool. The stark contrast between the numerical
effectiveness of their method and its mathematical ill-posedness, see
\cite{Kic97}, spurred significant subsequent research in mathematics
and image processing. A number of mathematical ``fixes'' have been
proposed over the past decades. It is referred to \cite{G131} for an
overview. Of relevance for this article is the fractional derivatives'
based regularization proposed in \cite{GL07}. While it is well-posed
as a quasi-linear parabolic equation, it appears so only in a smooth
context (i.e. for smooth enough initial data). Characteristic
functions or linear combinations thereof are, however, of extreme
interests in applications and mathematical results in functions spaces
which contain them are desirable. As, even in the 
corresponding linear case, uniqueness may fail to hold (see \cite{G161}
for an illustration), the careful choice of functional setting is
paramount. It has indeed been impossible thus far to identify the
appropriate concept of weak solution yielding well-posedness for a
class of initial data large enough to include characteristic functions
of smooth sets. Allowing for non-smooth initial data readily leads to
degenerate parabolic equations. The precise degeneration type,
however, depends on the exact properties of the chosen non-smooth
initial datum. The construction of a unique solution 
proposed here is therefore based on the use of function spaces defined
around a single singular function (in order to fix the degenration
type) and of recently developed results for parabolic equations on
singular Riemannian manifolds which provide a tool for analyzing
degenerate parabolic equations with fixed degeneration;
see \cite{Ama13, Ama13b, Shao15, Shao1502}.
While the results of this paper do not resolve the general
uniqueness/non-uniqueness question, they appear to be the first
delivering non-trivial existence results of solution to non-smooth
initial data and uniqueness in a restricted class of functions which
share a common singularity.

The remainder of the paper is organized as follows: Results about
maximal regularity for parabolic equations and weighted function spaces
are presented in Section 2. Local well-posedness of the nonlinear
model is shown in Section 3 and global well-posedness is established
in Section 4 by means of the principle of linearized stability for
small perturbations of the non-smooth initial datum. The main results
are formulated in Theorems \ref{S3.1: PM-thm-ID}, \ref{S3.2:
  PM-thm-2D}, and \ref{S4: Global existence}. 

\subsection{Notations}
For $s\geq 0$ and $p\in(1,\infty)$, we denote by $\F^s(\mathbb{R}^N)$
the function spaces obtained by replacing $\F$ by $W_p$ or $BC$. If
$\Q$ is the $N$-dimensional unit cube, the spaces $\F^s_\pi(\Q)$
are the corresponding subspaces consisting of periodic functions with
periodicity box given by $\Q$.

Given any topological set $U$, $\mathring{U}$ denotes the interior of
$U$. 

For any two Banach spaces $X,Y$, $X\doteq Y$ means that they are equal
in the sense of equivalent norms. The notation $\Lis(X,Y)$ stands for
the set of all bounded linear isomorphisms from $X$ to $Y$. 

The symbol $\sim$ always denotes Lipschitz equivalence. We write $\Nz=\N\setminus\{0\}$.

\section{\bf Maximal $L_p$-Regularity in a Weighted $L_p$-Framework}
\subsection{Transforming the problem onto the torus}
Let $N=1,2$, define $\Q=[-1,1)^N$, and consider the following problem: 
\begin{equation}\label{P-M eq}\begin{cases}
 \partial_t u -\div \bigl(  \alpha_\varepsilon (u)  \nabla u \bigr) =0  &\text{in
 }\Q\times (0,\infty), \\u&\text{periodic,}\\
 u(0)=u_0  &\text{in }\Q ,\end{cases}
\end{equation}
where $\alpha _\varepsilon(u)=\bigl[1+|\nabla^{1-\varepsilon}u|^2
\bigr]^{-1}$ and $\varepsilon\in (0,1)$. A precise definition of the
fractional derivative appearing in the nonlinear coefficient
$\alpha_\varepsilon$ will be given in Section 3.

We shall be interested in non-smooth initial data $u_0$ for which
$\alpha_\varepsilon(u_0)$ vanishes on a $C^3$-submanifold $\Gamma\subset
\mathring{\mathsf{Q}}^N$ of codimension 1 (which may not be connected).
For $\delta$ sufficiently small, we can always choose a
$2\delta$-tubular neighborhood $\mathscr{U}_{2\delta} \subset\subset 
\mathring{\mathsf{Q}}^N$ of any such $\Gamma$, even if $\Gamma$ has
merely $C^2$ boundary. Define $\d\in C^3(\Q\setminus\Gamma, (0,1])$
by
\begin{align}\label{S2.1: sing func}
 \d(x)=\begin{cases}
  {\sf dist}(x, \Gamma),\quad &\text{in }\sU\setminus \Gamma;\\
  1,&\text{in }\Q\setminus \mathscr{U}_{2\delta}.\end{cases}
\end{align}
and observe that ${\sf dist}(x, \Gamma)$ is well-defined and $C^3$ for
$\delta$ sufficiently small.

Considering $x_1,x_2\in \R^N$ to be equivalent if $x_1-x_2=2m$ for
some $m\in\mathbb{Z}^N$, let $\phi$ be the projection mapping taking
$x\in \R^N$ to its equivalence class. It clearly holds that
$\phi(\Q)=\TN$, where $\TN$ is the $N$-dimensional torus.

Throughout the rest of this paper, unless stated otherwise, we always
assume that 
\begin{itemize}
\item $s\geq 0$, $k\in\Nz$, $1<p\leq \infty$ and $\vartheta\in \R$.
\item $\F=W_p$ for $1<p<\infty$, or $\F=BC$.
\end{itemize}

\begin{remark}
\label{S2: RMK torus-cube}
If we equip $\TN$ with the metric $\phi_* g_N$, where $g_N$ is the
$N$-dimensional Euclidean metric on $\Q$ and $\phi_*=\bigl[(\phi\big
|_{\Q}) ^{-1}\bigr]^*$, 
i.e. $\phi_* g_N$ is the pullback metric along $(\phi\big
|_{\Q}) ^{-1}$,
then $(\mathbb{T}^N, \phi_* g_N )$ is a closed smooth manifold.
Therefore, any periodic function space $\F^s_\pi(\Q)$ defined on
$(\Q,g_N)$ is isomorphic to the corresponding $\F^s(\mathbb{T}^N)$
defined on $(\mathbb{T}^N, \phi_* g_N)$.  So all well-established
function space theory results, such as those pertaining to
interpolation and to lifting properties, transfer to the spaces
$\F^s_\pi(\Q)$. See, for instance, \cite[Chapter 7]{Trib92} for more
details on function space theory on closed manifolds.
\end{remark}

We let $\GT=\phi(\Gamma)$ and set
$$
 (M,g)=(\TN \setminus \GT, \phi_* g_N|_{\TN \setminus \GT}). 
$$
Denote the metrics $g_N$ and $\phi_* g_N$ by $(\cdot|\cdot)$ and
$(\cdot|\cdot)_g$, and the norms induced by $g_N$ and
$\phi_* g_N$ by $|\cdot|$ and $|\cdot|_g$, respectively. 

As long as it causes no confusion, we will denote the usual
covariant derivative, divergence, and Laplacian on $(\Q,g_N)$ as well as their restrictions
to $(\QG,g_N)$ by $\nabla$, $\div$ and $\Delta$ respectively. 
Similarly, $\nabla_g$, $\div_g$ and $\Delta_g$ will denote their
counterparts on both $(\mathbb{T}^N, \phi_* g_N)$ and $(\M,g)$.

Now problem \eqref{P-M eq} can be equivalently stated as
\begin{equation}\label{P-M-T eq}\begin{cases}
 \partial_t u -\div_g ( \alpha _\varepsilon(u)  \nabla_g u)=0
 &\text{in }\TN\times (0,\infty), \\ 
 u(0)=u_0  &\text{in }\TN.
\end{cases}\end{equation}
Here it is understood that $\alpha _\varepsilon(u)= \phi_*\alpha
_\varepsilon(\phi^*u)$. 

\subsection{Periodic weighted function spaces}

Note that the function defined by
\begin{equation}\label{S2.1: sing func-2}
 \rho(x)=\d(y),\: y\in \phi^{-1}(x)\cap \Q,
\end{equation}
is well-defined on $\M$ and satisfies $\rho\in C^3(\M,(0,1])$.  We
will begin with the definition of weighted function spaces on
$(\M,g)$ (see \cite{Ama13, AmaAr}) in order to derive the definition
of the corresponding weighted periodic function spaces on $\QG$.

Given an arbitrary finite dimensional Hilbert space $X$, denote its
inner product by $(\cdot|\cdot)_X$. The weighted Sobolev space of
$X$-valued functions $W^{k,\vartheta}_p(\M, X)$ is defined as the
completion of $\mathcal{D}(\M ,X)$, the space of $X$-valued
test-functions, with respect to the norm 
$$
 \|\cdot\|_{k,p;\vartheta}:u\mapsto(\sum_{i=0}^k\|\rho^{\vartheta+i}
 |\nabla_g^i u|_g\|_p^p)^{\frac{1}{p}},
$$
with the understanding that $L^\vartheta_p(\M,X)=W^{0,\vartheta}_p(\M
,X)$ and that $\nabla_g^{i+1} u:=\nabla_g\circ\nabla_g^i u$.
The weighted Sobolev-Slobodeckii spaces are defined as
\begin{equation*}
 W^{s,\vartheta}_p(\M ,X):=\Bigl(L^{\vartheta}_p(\M
 ,X),W^{k,\vartheta}_p(\M,X)\Bigr)_{s/k,p}, 
\end{equation*}
for $ s\in \R_+\setminus\N$, $k=[s]+1$. Here
$(\cdot,\cdot)_{\theta,p}$ is the standard real interpolation method
\cite[Chapter I.2.4.1]{Ama95}. Define
\begin{equation}\label{S2.2: def of BC}
BC^{k,\vartheta}(\M,X):=\Bigl(\{u\in{C^k(\M,X)}:\|u\|_{k,\infty;\vartheta}<\infty\}
 ,\|\cdot\|_{k,\infty;\vartheta}\Bigr),
\end{equation}
where $\|u\|_{k,\infty;\vartheta}:=\max_{0\leq i\leq
  k}\|\rho^{\vartheta+i}|\nabla_g^i u|_g\|_\infty$. 
\begin{remark}\label{S2: M-sing mnfd}
Note that $(\M,g)$ is an incomplete manifold. Indeed, by
  \cite[Lemma~3.4]{Ama14} and \cite[Proposition~12]{Shao1503},
$(\M,g;\rho)$ can be seen as a $C^2$-singular 
manifold. It follows that the weighted function spaces introduced above
are all well-defined for $(\M,g)$ (cf. \cite{Ama13, AmaAr, Shao15}). 
The properties of weighted function spaces defined on  $C^2$-singular
manifolds established in the cited references are all inherited by the
weighted function spaces $\F^{s,\vartheta}(\M,X)$.
\end{remark}
We can define periodic weighted function spaces on $(\QG, g_N)$ in the
same manner just by replacing the weight function $\rho$, $\nabla_g$
and $|\cdot|_g$ by $\d$, $\nabla$ and $|\cdot|$, respectively. We
denote these spaces by $\F^{s,\vartheta}_\pi(\QG,X)$.
By the identification 
\begin{equation}\label{S2: T=Q}
 \F^{s,\vartheta}_\pi(\QG,X)\doteq \phi^*\F^{s,\vartheta}(\M,X),
\end{equation}
the space $\F^{s,\vartheta}_\pi(\QG,X)$ enjoys the same properties as
$\F^{s,\vartheta}(\M,X)$. For notational brevity, we still denote the
norms of the weight function spaces $\F^{s,\vartheta}_\pi(\QG,X)$ by
$\|\cdot\|_{k,p;\vartheta}$ and $\|\cdot\|_{k,\infty;\vartheta}$,
respectively.
\begin{lem}\label{S2: nabla}
Let $s\geq 0$ for $\F=W_p$ or $s\in \N$ for $\F=BC$ and $\vartheta\in
\mathbb{R} $, then it holds that
\begin{itemize}
\item[(i)] $\nabla\in\L\Bigl(\F^{s,\vartheta}_\pi(\QG,
  \R),\F^{s-1,\vartheta+1}_\pi(\QG,\R^{N})\Bigr)$.
\item[(ii)]
  $\div\in\L\Bigl(\F^{s,\vartheta}_\pi(\QG,\R^N),\F^{s-1,\vartheta+1}_\pi(\QG,\R)\Bigr)$.
\end{itemize}
\end{lem}
\begin{proof}
(i) follows from \cite[Theorem~7.5]{Ama13}, \eqref{S2.2: def of BC}
and  \eqref{S2: T=Q}. Note that the definitions of weighted function spaces in this 
article are slightly different from those in \cite{Ama13}.
(ii) is a consequence of \cite[Propositions~2.5, 2.8]{Shao15}.  
The difference in the weights between this lemma and \cite[Proposition~2.8]{Shao15} 
is due to the fact that, here, the divergence operator acts on cotangent bundle,
while, in the cited reference, it acts on the tangent bundle.
\end{proof}

\begin{lem}
\label{S2: change of wgt}
For $\vartheta^\prime\in\R$ and $s,\vartheta$ as in Lemma \ref{S2:
  nabla}, we have that 
$$
 [u\mapsto \rho^{\vartheta}u] \in
 \Lis\Bigl(\F^{s,\vartheta^\prime+\vartheta}_\pi(\QG,
 X),\F^{s,\vartheta^\prime}_\pi(\QG, X)\Bigr).
$$ 
\end{lem}
\begin{proof}
See \cite[Propositions~2.4]{Shao15} and \eqref{S2: T=Q}.
\end{proof}

\begin{lem}
\label{S2: pointwise mul}
Let $s\leq k\in\Nz$ and  $\vartheta_i\in \R$ with $i=0,1$. $[(u,v)\mapsto  (u|v)_X]$ is a continuous
bilinear map in each of the following functional settings
\begin{multline*}
 W_{p,\pi}^{s,\vartheta_0}(\QG,X)\times
 BC^{k,\vartheta_1}_\pi(\QG,X)\to 
 W_{p,\pi}^{s,\vartheta_0+\vartheta_1 }(\QG)\text{ or}\\ 
 BC^{k,\vartheta_0 }_\pi(\QG,X)\times BC^{k,\vartheta_1 }_\pi(\QG,X)
 \to BC^{k,\vartheta_0+\vartheta_1 }_\pi(\QG).
\end{multline*}
\end{lem}
\begin{proof}
This follows from \cite[Theorem~13.5]{AmaAr} and \eqref{S2: T=Q}.
\end{proof}

\begin{lem}\label{S2: interpolation}
Suppose that $k_i\in \N$, $\vartheta_i\in \R$ with $i=0,1$,
$0<\theta<1$ and $k_0<k_1$. Then 
$$
 \Bigl(W^{k_0,\vartheta_0}_{p,\pi}(\QG, X),
 W^{k_1,\vartheta_1}_{p,\pi}(\QG, X)\Bigr)_{\theta,p} \doteq
 W^{k_\theta,\vartheta_\theta}_{p,\pi}(\QG, X),
$$ 
where $\xi_\theta:=(1-\theta)\xi_0+\theta \xi_1$ for any
$\xi_0,\xi_1\in \R$ and the case $k_\theta\in \N$ needs to be
excluded.
\end{lem}
\begin{proof}
It follows from \cite[Theorem~8.2(i), formulas (8.3), 
(21.2)]{AmaAr} and \eqref{S2: T=Q}.
\end{proof}

\begin{prop}\label{S2: Sobolev embedding}
Suppose that  $s>k+\frac{N}{p}$ and $\vartheta\in\R$. Then
$$
 W_{p,\pi}^{s,\vartheta}(\QG ,X) \hookrightarrow
 BC^{k,\vartheta+\frac{N}{p}}_\pi(\QG ,X).
$$
\end{prop}
\begin{proof}
See \cite[Theorem~14.2(ii)]{Ama13} and \eqref{S2: T=Q}.
\end{proof}

\subsection{Maximal Regularity of Type $L_p$}

In this subsection, we will state some preliminary concepts and results
of maximal $L_p$-regularity for differential operators and their
application to quasi-linear parabolic equations. The reader is
referred to \cite{Ama95}, \cite{ArenGrohNage86}, and
\cite{DenHiePru03} for more details about these concepts.

We consider the following abstract Cauchy problem 
\begin{equation}
\label{S2.2: Cauchy problem}
\left\{\begin{aligned}
\partial_t u(t) +\cA u(t) &=f(t), &&t\geq 0\\
u(0)&=0 . &&
\end{aligned}\right. 
\end{equation}
For $\theta\in (0,\pi]$, the open sector of angle $2\theta$ is denoted
by 
$$
 \Sigma_\theta:= \{\omega\in \mathbb{C}\setminus \{0\}: |\arg
 \omega|<\theta \}.
$$
\begin{definition}
Let $X$ be a complex Banach space, and $\cA$ be a densely defined
closed linear operator in $X$ with dense range. $\cA$ is called
sectorial if $\Sigma_\theta \subset \rho(-\cA)$ for some $\theta>0$
and 
$$
 \sup\{|\mu(\mu+\cA)^{-1}| : \mu\in \Sigma_\theta \}<\infty.
$$
The class of sectorial operators in $X$ is denoted by $\S(X)$.
The spectral angle $\phi_\cA$ of $\cA$ is defined by
$$
\phi_\cA:=\inf\{\phi:\, \Sigma_{\pi-\phi}\subset \rho(-\cA),\, 
\sup\limits_{\mu\in \Sigma_{\pi-\phi}} |\mu(\mu+\cA)^{-1}|<\infty. \}. 
$$ 
\end{definition}

\begin{definition}
Assume that $X_1\overset{d}{\hookrightarrow}X_0$ is some densely
embedded pair of Banach spaces. Suppose that $\cA\in \S(X_0)$ with
$dom(\cA)=X_1$. Then, the Cauchy problem \eqref{S2.2: Cauchy problem}
is said to have maximal $L_p$-regularity if, for any 
$$
 f\in L_p(\R_+,X_0),
$$
equation \eqref{S2.2: Cauchy problem} has a unique solution
$$
 u\in L_p(\R_+, X_1) \cap H^1_p(\R_+, X_0) .
$$
We denote this by 
$$
 \cA\in \mathcal{MR}_p(X_1, X_0).
$$
\end{definition}

Maximal regularity theory has proven a powerful tool in the
theory of nonlinear parabolic equations. We will apply it to the study
of problem~\eqref{P-M eq}. To this end, let us consider the following
abstract evolution equation 
\begin{equation}\label{S2.2: evolution eq}
\begin{cases}
 \partial_t u +\cA(u) u =f(u), &t\geq 0,\\
 u(0)=0,&
\end{cases} 
\end{equation}
in $X_0$. We have the following existence and uniqueness result for
equation~\eqref{S2.2: evolution eq}. 
\begin{theorem}
\label{S2.2: Thm-MR}
\cite[Theorem~2.1]{CleLi93}
Let $1<p<\infty$ and $X_1\overset{d}{\hookrightarrow}X_0$ be a densely
embedded pair of Banach spaces. Setting $X_{1/p}:=(X_0,X_1)_{1-1/p,p}$,
suppose that $U\subset X_{1/p}$ is open and  that $\cA,f$ satisfy 
$$
 (\cA, f)\in C^{1-}\bigl(U,\mathcal{MR}_p(X_1,X_0)\times X_0\bigr)
$$
Then for every $u_0\in U$, there exist $T=T(u_0)>0$ and a unique
solution of \eqref{S2.2: evolution eq} on $J=[0,T]$ with
$$
 u\in L_p(J, X_1)\cap H^1_p(J, X_0).
$$
\end{theorem}
\section{\bf Local well-posedness of the Nonlinear Model}

In this section, we will establish the existence and uniqueness of
solutions to \eqref{P-M eq}. 
The precise definition of the fractional gradient $\nabla^{1-\varepsilon}$ in
the one and two dimensional cases will be stated separately in the
following two subsections. In order to allow for non-smooth initial
data for \eqref{P-M eq} and the corresponding degeneration in the
diffusion coefficient they cause, it is necessary to resort to
weighted space. We put 
$$
 E_0:=L^{\vartheta+2\varepsilon}_{p,\pi}(\QG), \quad
 E_1:=W^{2,\vartheta}_{p,\pi}(\QG),
$$
and 
$$
E_{\frac{1}{p}}:=(E_0,E_1)_{1-1/p,p}=W^{2-2/p,\vartheta+\frac{2\varepsilon}{p}}_{p,\pi}(\QG).
$$
Throughout the rest of this section, we always assume
\begin{equation}\label{S3: ASP: vartheta & p}
\vartheta\leq -2,\: p>\max\{ \frac{N+2}{\varepsilon} ,
-\frac{N+2\varepsilon}{\vartheta}\},\:\varepsilon\neq \frac{1}{2}. 
\end{equation}
or
\begin{equation}\label{S3: ASP: vartheta & p-2}
\vartheta=-2\varepsilon,\: p>\max\{ \frac{2N+2}{\varepsilon},
\frac{4N+5}{2}\},\: \varepsilon>1-\frac{1}{2p}. 
\end{equation}
Conditions \eqref{S3: ASP: vartheta & p} and \eqref{S3: ASP: vartheta
  & p-2} above are imposed in order for 
technically necessary embeddings to be valid. Notice that the first
condition allows for more freedom in the choice of $\varepsilon$, whereas the
second will make it possible to obtain stronger results (see Section
4). 

If \eqref{S3: ASP: vartheta & p} holds, then it is
not hard to verify by the definition of $W^{k,\vartheta}_p(\QG)$ in
Section~2.2 and the choice of $\vartheta$ that
\begin{equation*}
E_0\hookrightarrow L_{p,\pi}(\Q)\text{ and }
E_1\hookrightarrow W^2_{p,\pi}(\Q). 
\end{equation*}
Interpolation theory implies that
\begin{equation}
\label{S3: Trace space embedding}
E_{\frac{1}{p}}\hookrightarrow W^{2-2/p}_{p,\pi}(\Q).
\end{equation}
We define $\R_\Gamma$ to be the set of functions which are a constant
on each component of $\QG$. 

\begin{lem}\label{S3: reduction p/2}
Assume that \eqref{S3: ASP: vartheta & p-2} is satisfied. Then
$$
 E_1\hookrightarrow W^2_{p/2,\pi}(\Q).
$$
\end{lem}
\begin{proof}
First note that for $p$ satisfying \eqref{S3: ASP: vartheta & p-2}, by
Proposition~\ref{S2: Sobolev embedding}, one has that
$$
E_1 \hookrightarrow BC^{1, \vartheta +\frac{N}{p}}_\pi(\QG).
$$
Therefore,  any $u\in E_1$ admits a smooth trace $\gamma_\Gamma(u)$
and $\gamma_\Gamma(u)=0$; similarly $\gamma_\Gamma(|\nabla u|)=0$, in
view of the assumptions on the parameter $\vartheta$. 
These estimates imply that, on each component of $\QG$, $u$ and $\nabla u$ 
admit unique continuous extensions onto $\Gamma$ and thus that $E_1
\hookrightarrow BC^1_\pi(\Q)$.

Pick $q<p$. First, it is clear that
$E_1\hookrightarrow L_{q,\pi}(\Q) $ since, by definition, 
\begin{align*}
 \int_{\Q} |\nabla^2 u|^q\, dx &\leq \int_{\Q}
 |\d^{2\varepsilon-2}\d^{2-2\varepsilon}\nabla^2u|^q\, dx\\
 &\leq  \Big[\int_{\Q}|\d^{2\varepsilon-2}|^{\frac{qp}{p-q}}\,
   dx\Big]^{\frac{p-q}{p}}
  \Big[\int_{\Q} |\d^{2-2\varepsilon}\nabla^2 u|^p\, dx\Big]^{q/p}.
\end{align*}
To make the first term on the right hand side of the inequality
finite, it suffices to require that
$(2\varepsilon-2)\frac{qp}{p-q}>-1$. This is clear for
  $N=1$ where the singularity is at isolated points, whereas
for $N=2$ it follows from the fact that the singularity is along a
smooth curve. The above parameter inequality is equivalent to 
$$
 \varepsilon>1-\frac{1}{2q}+ \frac{1}{2p}.
$$
Taking $q=p/2$ yields $\varepsilon>1-\frac{1}{2p}.$  	
\end{proof}
The assumption $p>\frac{4N+5}{2}$ in \eqref{S3: ASP: vartheta & p-2}
guarantees that $p>\frac{2N+2}{\varepsilon}$
and $\varepsilon>1-\frac{1}{2p}$ do not conflict.

\subsection{One Dimensional Case} 

Since we are working with periodic functions on $\Qq =[-1,1)$, we
can define $\nabla^{1-\varepsilon}=\partial^{1-\varepsilon}$ by means
of Fourier series 
$$
 \partial^{1-\varepsilon}u
 :=|\partial|^{-\varepsilon}u^\prime:=\mathcal{F}^{-1}{\rm
  diag}\{\frac{i\pi k}{|k|^\varepsilon}\}\mathcal{F} u,
$$
where $\mathcal{F}$ denotes the Fourier transform and ${\rm
  diag}\{\frac{i\pi k}{|k|^\varepsilon}\}$  denotes the multiplication
operator (in Fourier space) by the function $[k\mapsto\frac{i\pi
  k}{|k|^\varepsilon}]$.
\begin{lem}\label{S3.1: Sing ker of frac grad-ID}
$$
 \partial^{1-\varepsilon}u(x)=\int_{\Qq} c_\varepsilon
 \frac{u^\prime(y)}{|x-y|^{1-\varepsilon}}\, dy +\int_{\Qq}
 h_\varepsilon(x-y)  u^\prime(y)\, dy,
$$
for some constant $c_\varepsilon>0$ and $h_\varepsilon\in C^\infty$.
\end{lem}
\begin{proof}
By definition, one has that
$$
 \partial ^{1-\epsilon}u(x)=\int_{\Qq}G_\epsilon(x-y)u'(y)\,dy, 
$$
with
$$
 \widehat G_\epsilon (k)=\frac{1}{|k|^\epsilon},\: k\in
 \mathbb{Z}^*:=\mathbb{Z}\setminus\{ 0\}.
$$
This means that
$$
 G_\epsilon(x)=\sum _{k\in\mathbb{Z}^*}\frac{1}{|k|^\epsilon}e^{\pi
 ikx}=\sum _{k\in\mathbb{Z}}\frac{\eta(k)}{|k|^\epsilon}e^{\pi ikx},
$$
where $\eta\in C^\infty(\mathbb{R})$ is a cut-off function with
$$
 \eta(x)=\begin{cases} 0,&|x|\leq 1/4,\\ 1,&|x|\geq 1/2.\end{cases}
$$
Notice that the Poisson's summation formula
  (\cite[p. 362]{Stein93}) yields
$$
 G_\epsilon(x)=\sum
 _{k\in\mathbb{Z}}\frac{\eta(k)}{|k|^\epsilon}e^{\pi
   ikx}=g_\epsilon(x)+\sum 
 _{k\in \mathbb{Z}^*}g_\epsilon(x+k),\: x\in (-1,1),
$$
where $g_\epsilon=\mathcal{F}\bigl(\eta |\cdot |^{-\epsilon}\bigr)$ is rapidly
decreasing (faster than the reciprocal of any polynomial) as the
Fourier transform of a smooth function, and satisfies 
$$
 g_\epsilon=c_\epsilon |\cdot|^{\epsilon -1}+\mathcal{F}\bigl([\eta-1] |\cdot
 |^{-\epsilon}\bigr),\: x\in \mathbb{R}, 
$$
where the second addend is a smooth function as the Fourier transform
of a compactly supported function. Here the fact that $\mathcal{F}\bigl( |\cdot
|^{-\epsilon}\bigr)=c_\epsilon |\cdot |^{\epsilon -1}$ was also used. Combining
everything together yields the claimed decomposition with
$$
 h_\epsilon=\mathcal{F}\bigl([\eta-1] |\cdot|^{-\epsilon}\bigr)+\sum _{k\in
   \mathbb{Z}^*}g_\epsilon(\cdot+k).
$$
\end{proof}

\begin{lem}\label{S3.1: fractional grad-1D}
Assume that \eqref{S3: ASP: vartheta & p} or 
\eqref{S3: ASP: vartheta & p-2} is satisfied. Then
$$
 \partial^{1-\varepsilon}\in \L(E_{\frac{1}{p}}, BC^1_\pi(\Qq)).
$$
\end{lem}
\begin{proof}
First note that $\F^s_\pi(\Qq)\doteq \phi^* \F^s(\mathbb{T})$, so 
function space theory on compact manifolds applies; see \cite[Chapter
7]{Trib92}.\\
(i) If \eqref{S3: ASP: vartheta & p} is assumed to hold, from \eqref{S3:
  Trace space embedding}, we can infer that 
$$
 \partial^{1-\varepsilon} \in \L(E_{\frac{1}{p}},
 W^{1+\varepsilon-2/p}_{p,\pi}(\Qq)),
$$
since the mapping properties of
$\partial^{1-\varepsilon}$ readily follow for the Bessel potential
spaces (which can be defined in terms of decay properties of Fourier
coefficients) and, then for Sobolev-Slobodeckii spaces as well, by
interpolation.
Now the statement follows from Sobolev embedding and from
$\displaystyle p>\frac{3}{\varepsilon}$ in \eqref{S3: ASP: vartheta & p}.\\
(ii) If, instead, we assume \eqref{S3: ASP: vartheta & p-2}, then, 
Lemma~\ref{S3: reduction p/2} and interpolation theory imply that
$$
 E_{\frac{1}{p}}\hookrightarrow \Bigl(W^2_{p/2,\pi}(\Qq),
 L_{p/2,\pi}(\Qq)\Bigr)_{1-1/p,p}\doteq B^{2-2/p}_{p/2,p,\pi}(\Qq) .
$$
Thus we infer that
$$
 \partial^{1-\varepsilon} \in \L(E_{\frac{1}{p}},
 B^{1+\varepsilon-2/p}_{p/2,p,\pi}(\Qq)).
$$
Now embedding theorems for Besov spaces and
$\displaystyle p>\frac{4}{\varepsilon}$ in \eqref{S3: ASP: vartheta & p-2}
complete the proof.
\end{proof}
In dimension $1$, we are interested in initial data that are
close to piecewise constant functions in some proper topology. 

It suffices to take the following function $H$ as a generic
representative of piecewise constant functions 
\begin{align}\label{S3.1: def H}
H(x)=\chi_{(-1/2,1/2)}(x)=
\begin{cases}
1,\quad &x\in(-1/2,1/2),\\
0, &|x|\geq 1/2.
\end{cases}
\end{align}
This means that we choose $\Gamma=\{\pm 1/2\}$ and $H\in \R_\Gamma$.

\begin{prop}\label{S3.1: frac der of H}
The function $\partial^{1-\varepsilon} H$ satisfies
$$
 \partial^{1-\varepsilon} H(x)
 =c_\varepsilon[\frac{1}{|x+1/2|^{1-\varepsilon}}-\frac{1}{|x-1/2|^{1-\varepsilon}}]
 + h_\varepsilon(x+1/2) - h_\varepsilon(x-1/2),\: x\in\Qq
$$
where $h_\varepsilon\in C^\infty$, for some constant $c_\varepsilon>0$.
\end{prop}
\begin{proof}
Using the kernel representation given in Lemma~\ref{S3.1: Sing ker of
  frac grad-ID} and the fact that $H^\prime=\delta_{-1/2}-
\delta_{1/2}$ yields that 
$$
 \partial^{1-\varepsilon} H(x)
 =c_\varepsilon[\frac{1}{|x+\frac{1}{2}|^{1-\varepsilon}} -
 \frac{1}{|x-\frac{1}{2}|^{1-\varepsilon}} ] + h_\varepsilon(x+1/2) -
 h_\varepsilon(x-1/2),\: x\in\Qq,
$$
and the claim follows.
\end{proof}

Taking $\d$ as in \eqref{S2.1: sing func},
there exists some $\cE>1$ such that
\begin{equation}\label{S3.1: est for H}
 1/\cE < \d^{1-\varepsilon} |\partial^{1-\varepsilon} H |<\cE,\quad
 \text{near }\Gamma. 
\end{equation}
We assume that the initial datum is of the form 
$$
 u_0= H+ w_0, \quad w_0\in E_{\frac{1}{p}}.
$$
\begin{remark}
A typical example of the perturbation term $w_0$ could be $\sin(64\pi x^2)$. 
\end{remark}
For any $w\in E_{\frac{1}{p}}$, we define 
$$
 \sA(w)u:= -\div( \alpha _\varepsilon(H+w) \nabla u).
$$ 
Recall that $\displaystyle \alpha
_\varepsilon(H+w):=\frac{1}{1+|\partial^{1-\varepsilon}(H+w)|^2}.$

We will apply the theory of $\mathscr{R}$-bounded operators to prove
that the operator $\sA(w)$ enjoys the property of maximal
$L_p$-regularity.
\begin{definition}
Let $X_1$ and $X_0$ be two Banach spaces. A family of operators
$\mathcal{T}\in \L(X_1,X_0)$ is called $\mathscr{R}$-bounded, if there
is a constant $C>0$ and $p\in [1,\infty)$ such that for each $N\in\N$,
$T_j\in \mathcal{T}$ and $x_j\in X_1$ and for all independent,
symmetric, $\{-1,1\}$-valued random  variables $\varepsilon_j$ on a
probability space $(\Omega,\mathcal{M},\mu)$ the inequality 
$$ 
  |\sum\limits_{j=1}^N \varepsilon_j T_j x_j|_{L_p(\Omega; X_0)} \leq
  C|\sum\limits_{j=1}^N \varepsilon_j  x_j|_{L_p(\Omega; X_1)} 
$$ 
is valid. The smallest such $C$ is called $\mathscr{R}$-bound of
$\mathcal{T}$. We denote it by $\mathscr{R}(\mathcal{T})$. 
\end{definition}

\begin{definition}
Suppose that $\cA\in \S(X)$. Then $\cA$ is called
$\mathscr{R}$-sectorial if there exists some $\phi>0$ such that 
$$
 \mathscr{R}_\cA(\phi):=\mathscr{R}\{\mu(\mu+\cA)^{-1}: \mu\in
 \Sigma_\phi \}<\infty.
$$
The $\mathscr{R}$-angle $\phi^R_\cA$ is defined by
$$
  \phi^R_\cA:=\inf\{ \theta\in (0,\pi): \mathscr{R}_\cA(\pi-\theta)
  <\infty \}.
$$
The class of $\mathscr{R}$-sectorial operators in $X$ is denoted by
$\RS(X)$. 
\end{definition}

Let $R>0$ and $B_R:=\{w\in E_{\frac{1}{p}}: \|w\|_{E_{\frac{1}{p}}}<R\}$.
\begin{lem}
\label{S3.1: unif bd of alpha}
There exists a constant $C$ such that
$$\d^{1-\varepsilon}|\partial^{1-\varepsilon}(H+w)|<C,\quad w\in B_R,$$
and 
$$1/C< \d^{2\varepsilon-2} \alpha _\varepsilon(H+w)<C ,\quad w\in B_R.$$
\end{lem}
\begin{proof}
(i) It follows from Lemma~\ref{S3.1: fractional grad-1D} that
\begin{equation*}
|\partial^{1-\varepsilon}w| \quad\text{is uniformly bounded for }w\in B_R,
\end{equation*} 
and the boundedness of $\d^{1-\varepsilon}|\partial^{1-\varepsilon}
H|$ follows from \eqref{S3.1: est for H}. This proves the first
assertion.\\
(ii) We have that
\begin{align*}
\d^{2\varepsilon-2} \alpha _\varepsilon(H+w)=
  \frac{1}{\d^{2-2\varepsilon}+
  \d^{2-2\varepsilon}|\partial^{1-\varepsilon}(H+w)|^2}. 
\end{align*}
The first assertion implies the uniform lower bound of the second. It
follows from the expression for $\partial^{1-\varepsilon}H$ in
Proposition~\ref{S3.1: frac der of H} that, in a small enough
$\delta$-neighborhood ${\sU}$ of $\Gamma$, one has that
$$
 \d^{1-\varepsilon}|\partial^{1-\varepsilon}H|>\frac{c_\varepsilon}{2},
$$
where $c_\varepsilon$ is the constant in Proposition~\ref{S3.1: frac
  der of H}. 
Clearly, by the uniform boundedness of $|\partial^{1-\varepsilon}w|$
in $B_R$, it holds that 
\begin{align}\label{S3: crucial embedding}
 \d^{1-\varepsilon}(x)  w(x)\to 0,\quad \text{and}\quad
  \d^{2-\varepsilon}(x) \frac{d}{dx} w(x)\to 0 \quad \text{as}\quad
  x\to \Gamma. 
\end{align}
Choosing $\delta$ sufficiently small yields
$\displaystyle \d^{1-\varepsilon}|\partial^{1-\varepsilon}w|<\frac{c_\varepsilon}{4}$
inside $\sU$. Therefore, 
$$
 \d^{2-2\varepsilon}|\partial^{1-\varepsilon}(H+w)|^2>\frac{c^2_\varepsilon}{16}
 \quad \text{in }\sU,\quad w\in U_R.
$$
Outside $\sU$, $\d^{2-2\varepsilon}$ is bounded from below
by a positive constant. This
proves the uniform upper bound in the second assertion. 
\end{proof}

\begin{lem}\label{S3.1: der alpha}
 For each $w\in E_{\frac{1}{p}}$,
 $\displaystyle\Big|\d^{2\varepsilon-1} \frac{d}{dx}\alpha
 _\varepsilon(H+w)\Big |\sim{\bf 1}$  in a $\delta$-neighbourhood $\sU$ of
 $\Gamma$.  
\end{lem}
\begin{proof}
Let $u=H+w$ and observe that
\begin{align*}
\frac{d}{dx} \alpha _\varepsilon(u)= -2\alpha
  _\varepsilon^2(u) \partial ^{1-\varepsilon}u
  \frac{d}{dx}\partial^{1-\varepsilon}u.
\end{align*}
An easy computation and Proposition~\ref{S3.1: frac der of H} show
that 
\begin{equation*}
 \d^{2-\varepsilon}|\frac{d}{dx}\partial^{1-\varepsilon}H |\sim {\bf
   1}. 
\end{equation*}
By \eqref{S3: crucial embedding}, in a sufficiently small  $\delta$-neighbourhood $\sU$ of
$\Gamma$, we have that
$$
 \d^{1-\varepsilon}|\partial^{1-\varepsilon}u|\sim {\bf
   1},\quad\d^{2-\varepsilon}|\frac{d}{dx}\partial^{1-\varepsilon}u|
  \sim {\bf 1}.
$$
In combination with Lemma~\ref{S3.1: unif bd of
  alpha}, this yields 
\begin{equation}\label{S3.1: est for u_0-3}
 |\d^{2\varepsilon-1}\frac{d}{dx}\alpha _\varepsilon(u)|\sim {\bf 1},
 \quad \text{in }\sU. 
\end{equation}
\end{proof}

\begin{lem}\label{S3.1: 2nd der of alpha}
There exists a constant $C$ such that 
$$
1/C< \sg(1-2\varepsilon)\d^{2\varepsilon}\frac{d^2}{dx^2}\alpha
 _\varepsilon(H)<C
$$
in a $\delta$-neighbourhood $\sU$ of $\Gamma$.
\end{lem}
\begin{proof}
Direct computations show
\begin{multline*}
\frac{d^2}{dx^2}\alpha _\varepsilon(H)=
   \alpha
  _\varepsilon(H)^3
  |\partial^{1-\varepsilon} H|^2
  \big [6|\frac{d}{dx} \partial^{1-\varepsilon} H|^2    
  -2\partial^{1-\varepsilon}H \frac{d^2}{dx^2}\partial^{1-\varepsilon}H 
  \big] \\  
  -2\alpha
  _\varepsilon(H)^3\big [
  \partial^{1-\varepsilon}H \frac{d^2}{dx^2}\partial^{1-\varepsilon}H 
  +|\frac{d}{dx} \partial^{1-\varepsilon} H|^2  \big] 
  .\qquad
\end{multline*}
By Proposition~\ref{S3.1: frac der of H},  one can verify that
$$
 \d^{3-\varepsilon}\frac{d^2}{dx^2}\partial^{1-\varepsilon}H \sim {\bf
   1},
$$
and 
$$
6|\frac{d}{dx} \partial^{1-\varepsilon} H|^2    
  -2\partial^{1-\varepsilon}H \frac{d^2}{dx^2}\partial^{1-\varepsilon}H
  \sim 
  2(1-\varepsilon)(1-2\varepsilon)\d^{2\varepsilon-4},\quad \text{near }\Gamma.
$$
Note that, by the previous estimates and  Lemma~\ref{S3.1: der alpha},
it holds that
$$
\d^{2\varepsilon}\alpha
  _\varepsilon(H)^3\big |
  \partial^{1-\varepsilon}H \frac{d^2}{dx^2}\partial^{1-\varepsilon}H 
  +|\frac{d}{dx} \partial^{1-\varepsilon} H|^2  \big| \sim \d^{2-2\varepsilon}.
$$
Thus this term can be made arbitrarily small by shrinking $\sU$.
To sum up, in a sufficiently small $\delta$-tubular neighborhood $\sU$,
we have
\begin{align*}
 \sg(1-2\varepsilon)\d^{2\varepsilon}\frac{d^2}{dx^2}\alpha
 _\varepsilon(H)\sim |1-2\varepsilon|.
\end{align*}
This completes the proof.
\end{proof}

\begin{lem}\label{S3.1: alpha BC1}
$\alpha _\varepsilon(H)\in
BC^{2,2\varepsilon-2}_\pi(\QqG)$, and for each $w\in E_{\frac{1}{p}}$, we have
$$
\alpha _\varepsilon(H+w)\in
BC^{1,2\varepsilon-2}_\pi(\QqG).
$$ 
Moreover, for any $R>0$, 
$$[w\mapsto \alpha _\varepsilon(H+w)] \in C^\omega(B_R, BC^{1,
  2\varepsilon-2}_\pi(\QqG)),$$ 
  where $\omega$ is the symbol for real analyticity.
\end{lem}
\begin{proof}
By Lemma~\ref{S3.1: fractional grad-1D} and Proposition~\ref{S3.1: frac der of H}, we readily infer that
$$
\alpha _\varepsilon(H)\in C^2_\pi(\QqG) \quad \text{and}\quad \alpha _\varepsilon(H+w)\in C^1_\pi(\QqG).
$$ 
The rest of the proof for the first assertion follows from Lemmas~\ref{S3.1: unif bd of alpha}-\ref{S3.1: 2nd der of alpha}.

By the estimates in Lemmas~\ref{S3.1: unif bd of alpha} and \ref{S3.1: der alpha}, we already knew that
$$
\partial^{1-\varepsilon}(H+w)\in BC^{1,1-\varepsilon}(\QqG).
$$
Lemma~\ref{S2: pointwise mul} implies that
$$
 [w\mapsto \frac{1}{\alpha _\varepsilon(H+w)}]\in
 C^\omega(E_{\frac{1}{p}}, BC^{1,2-2\varepsilon}_\pi(\QqG)).
$$
The manifold $(\hat{\M}, \hat{g}):=(\M, g/\rho^2 )$ has bounded
geometry, and thus $BC^k$-function spaces are well defined. We denote
these spaces by $BC^k(\hat{\M})$. Note that the space 
$$
 BC^{1,0}_\pi(\QqG)\doteq \phi^* BC^1(\hat{\M}).
$$
See \cite[Section~4]{Ama13b}. Applying
\cite[Proposition~6.3]{ShaoSim13} to $BC^k$-functions and
in view of Lemma~\ref{S2: change of wgt}, we can infer that 
$$
 [w\mapsto \alpha _\varepsilon(H+w)] \in C^\omega(B_R, BC^{1,
  2\varepsilon-2}_\pi(\QqG)) 
$$
for any $R>0$.
\end{proof}

\begin{lem}\label{S3.1: approx H to H+w}
Assume that $w\in E_{\frac{1}{p}}$. Then for each $k$, there exists a sufficiently small 
$\delta_k$-neighbourhood $\sUk$ of $\Gamma$ such that
$$\|\alpha _\varepsilon(H)-\alpha _\varepsilon(H+w)\|_{BC^{1,2\varepsilon-2}(\bsUGk})\leq 1/k .$$ 
\end{lem}
\begin{proof}
We have
$$
\alpha _\varepsilon(H)-\alpha _\varepsilon(H+w)
= \alpha _\varepsilon(H)\alpha _\varepsilon(H+w)
\big[ \partial^{1-\varepsilon}(2H+w)\big]
\partial^{1-\varepsilon}w.
$$
Thus, by Lemma~\ref{S3.1: unif bd of alpha}, for some $C>0$
\begin{align*}
\d^{2\varepsilon-2}|\alpha _\varepsilon(H)-\alpha _\varepsilon(H+w)|\leq C \d^{1-\varepsilon} 
|\partial^{1-\varepsilon}w|.
\end{align*}
This term can be made arbitrarily small by shrinking the
neighbourhood $\sU$. The estimate for $\displaystyle \frac{d}{dx}\big[
\alpha _\varepsilon(H)-\alpha _\varepsilon(H+w)\big] $ follows
similarly by utilizing Lemmas~\ref{S3.1: fractional grad-1D},
\ref{S3.1: unif bd of alpha} and \ref{S3.1: der alpha}. 
\end{proof}

We can now establish the following maximal regularity property for the
operator $\sA(w)$ for every $w\in E_{\frac{1}{p}}$.
\begin{prop} \label{S3.1: MR-thm}
Let $1<p<\infty$ and $\varepsilon$ satisfy \eqref{S3: ASP: vartheta & p} or 
\eqref{S3: ASP: vartheta & p-2}. Then, for any $w\in
E_{\frac{1}{p}}$, the operator 
$$
 \sA(w)\in \mathcal{MR}_p(W^{2,\vartheta}_{p,\pi}(\QqG),
 L^{\vartheta+2\varepsilon}_{p,\pi}(\QqG)).
$$ 
\end{prop}
\begin{proof}
This theorem is a consequence of the work in \cite{Shao15,
  Shao1502}. We would like to refer the reader to these two papers for
more details, and thus only necessary explanations will be pointed out
here.

(i) For small $\delta>0$, by \cite[Theorem~1.6]{Ama14}, $(\bsUG,dx)$
is a  singular manifold.  

Lemmas~\ref{S3.1: unif bd of alpha}-\ref{S3.1: 2nd der of alpha} imply that
$$\alpha_\varepsilon(H)^{\frac{1}{2-2\varepsilon}}\in
BC^{2,-1}(\bsUG),\quad
\d\alpha_\varepsilon(H)^{\frac{1}{2-2\varepsilon}}\sim {\bf 1}. $$ 
Put $\displaystyle h=\sg(1-2\varepsilon)\log \alpha_\varepsilon(H)$. 
Then direct computations show
\begin{align*}
|\alpha_\varepsilon(H)^{\frac{1}{2-2\varepsilon}}\frac{d}{dx} h|=
  |\alpha_\varepsilon(H)^{\frac{2\varepsilon-1}{2-2\varepsilon}}
  \frac{d}{dx} \alpha_\varepsilon(H)|  
\sim  |\d^{2\varepsilon-1} \frac{d}{dx} \alpha_\varepsilon(H)|\sim {\bf 1},
\end{align*}
via Lemma~\ref{S3.1: der alpha},
and by Lemma~\ref{S3.1: 2nd der of alpha}
\begin{align*}
&\alpha_\varepsilon(H)^{\frac{2\varepsilon}{2-2\varepsilon}}
  \frac{d}{dx} (\alpha_\varepsilon(H) \frac{d}{dx} h)\\ 
=&\sg(1-2\varepsilon)
   \alpha_\varepsilon(H)^{\frac{2\varepsilon}{2-2\varepsilon}}
   \frac{d^2}{dx^2} \alpha_\varepsilon(H)\sim
   \sg(1-2\varepsilon)\d^{2\varepsilon} \frac{d^2}{dx^2}
   \alpha_\varepsilon(H) \sim {\bf 1}. 
\end{align*}
Therefore, the function $h$ satisfies conditions
($\sH_{2\varepsilon}1$) and ($\sH_{2\varepsilon}2$) defined in
\cite[Section~5.1]{Shao15} with $\lambda=2\varepsilon$ on
$(\bsUG,dx)$. This means that $(\bsUG,dx)$ with
$\alpha_\varepsilon(H)^{\frac{1}{2-2\varepsilon}}$ as a {\em
  singularity function} is  a singular manifold  satisfying
property $\sH_{2\varepsilon}$. The reader may refer to \cite{Shao15}
for more details.

The proof of \cite[Theorem~5.18]{Shao15} shows that  the operator
$-\sA(0)$  generates an analytic contraction strongly continuous
semigroup on $L^{2\varepsilon+\vartheta}_p(\bsUG)$ 
with 
$$D(\sA(0))\doteq \mathring{W}^{2,\vartheta}_p(\bsUG),\quad 1<p<\infty.
$$ 
Here for $\F\in\{BC, W_p\}$,  
$ \mathring{\F}^{s,\vartheta}(\bar{\mathscr{U}}_{\delta_k}\setminus
\Gamma) $ is defined as the closure of $\mathcal{D}(\sUG)$ in $
\F^{s,\vartheta}_\pi(\QqG) $. One can show that the semigroup
$\{e^{-t\sA(0)}\}_{t\geq 0}$ is positive by means of the same argument as in
step (iii) of the proof for \cite[Theorem~4.8]{Shao1502}.  

(ii) Let 
$$
X_0(\delta)=L^{2\varepsilon+\vartheta}_p(\bsUG),\quad X_1(\delta)=
\mathring{W}^{2,\vartheta}_p(\bsUG). 
$$
Now, following exactly the same argument as in step (iv) and (4.14) of
the proof for \cite[Theorem~4.8]{Shao1502}, one concludes that 
$$
\sA(0)\in \RS(X_0(\delta))\quad \text{with } \phi^R_{\sA(0)}<\pi/2.
$$
Moreover, by the definition of $\mathscr{R}$-bound, it is easy to
verify that for some $\theta>\pi/2$ 
$$\mathscr{R}\{\mu(\mu+\sA)^{-1}:\, \mu\in \Sigma_\theta\} \text{ is
  increasing in }\delta.$$ 
So is the norm $\|\sA^{-1}(0)\|_{X_0(\delta),X_1(\delta)}$. 

It follows from Lemmas~\ref{S2: nabla} and \ref{S3.1: approx H to H+w}
that, by shrinking $\delta$, we can always make 
$
\|(\sA(w)-\sA(0)) \sA^{-1}(0)\|_{\L(X_0(\delta))}
$
arbitrarily small. 
As a direct consequence of the perturbation theorem of
$\mathscr{R}$-sectorial operators,
cf. \cite[Proposition~4.2]{DenHiePru03}, we infer that 
$$
\sA(w)\in \RS(X_0(\delta))\quad \text{with } \phi^R_{\sA(w)}<\pi/2.
$$
The last step is to use a standard decomposition and gluing
procedure as in step (v)-(vii) of the proof for
\cite[Theorem~4.8]{Shao1502}, and we can prove that for some
$\omega\geq 0$ 
$$
\omega+\sA(w)\in \RS(L_{p,\pi}^{2\varepsilon+\vartheta}(\QqG))\quad \text{with }
\phi^R_{\sA(w)}<\pi/2. 
$$
Then the assertion follows from \cite[Theorem~4.4]{DenHiePru03}.
\end{proof}

Now we will apply Proposition \ref{S3.1: MR-thm} to proving
existence and uniqueness of solutions to equation
\eqref{P-M eq}. We first consider the problem  linearized in the
initial datum $H$. 
\begin{equation}\label{P-M eq-linear}
 \begin{cases}
   \partial_t u -\div\bigl(\alpha _\varepsilon(H)  \nabla u\bigr)=0
  &\text{in }\Qq\times (0,\infty), \\ 
  u&\text{periodic,}\\
  u(0)=H  &\text{in }\Qq .\end{cases}
\end{equation}
Clearly, $u^*\equiv H$ solves \eqref{P-M eq-linear}.

Then we look at the nonlinear problem
\begin{equation}\label{P-M eq-nonlinear}
 \begin{cases}
  \partial_t u -\div\bigl(\alpha _\varepsilon(u+u^*)  \nabla u\bigr)=0
  &\text{in }\Qq\times (0,\infty), \\
  u&\text{periodic,}\\
  u(0)=w_0  &\text{in }\Qq .
 \end{cases}
\end{equation}

Take $R>0$ so large that $w_0\in B_R$, then by
Lemmas~\ref{S2: nabla}, \ref{S2:
  pointwise mul}, \ref{S3.1: alpha BC1} and Proposition~\ref{S3.1: MR-thm},
\begin{equation}\label{S3.1: Reg-main}
\bigl[w\mapsto \div\bigl(\alpha _\varepsilon(H+w) \nabla
\cdot\bigr)\bigr]\in C^\omega\bigl(B_R,
\mathcal{MR}_p(E_1,E_0)\bigr).
\end{equation}
Hence the condition in Theorem~\ref{S2.2: Thm-MR} is satisfied. 
The same theorem implies the existence of a unique solution
$$
 \tilde{u}\in \bE_1(J):= L_p(J, E_1) \cap H^1_p(J, E_0)
$$
to \eqref{P-M eq-nonlinear}. We thus conclude that
$\hat{u}=\tilde{u}+u^*$ is a solution to \eqref{P-M eq} with initial
value $u_0=H+w_0$. 

We will show that $\hat{u}$ is indeed the unique solution in the class
$\bE_1(J)\oplus \R_\Gamma$, where 
$$
 \bE_1(J)\oplus \R_\Gamma:= \big\{u\in L_{1,loc}(J\times (\QqG)):
 u=u_1+u_2,\: u_1\in \bE_1(J),\, u_2\in\R_\Gamma\big\}. 
$$
Note that by \cite[Formula~(2.1)]{CleLi93} and \eqref{S3: ASP:
  vartheta & p} and \eqref{S3: ASP: vartheta & p-2} 
$$
\bE_1(J)\hookrightarrow C(J, E_{\frac{1}{p}}),\quad \text{and}\quad
\R_\Gamma \cap E_{\frac{1}{p}}=\{{\bf 0}_{\Qq\setminus\Gamma}\}.
$$
Indeed, by Proposition~\ref{S2: Sobolev embedding},
$E_{\frac{1}{p}}\hookrightarrow BC^{1,
  \vartheta+\frac{2\varepsilon+N}{p}}_\pi(\QqG)$. But
$p>-\frac{N+2\varepsilon}{\vartheta}$ in \eqref{S3: ASP: vartheta & p}
or $p>\frac{2N+2}{\varepsilon}$ in \eqref{S3: ASP: vartheta & p-2}
implies that  
$$
 u(x)\to 0 \quad \text{as}\quad x\to \Gamma,\quad u\in BC^{1,
   \vartheta+\frac{2\varepsilon+N}{p}}_\pi(\QqG).
$$
For any $u\in \bE_1(J)\oplus \R_\Gamma$, we have thus a unique
decomposition
$$
 u=u_1 +u_2 \quad \text{with}\quad u_1\in \bE_1(J),\quad u_2\in
 \R_\Gamma.
$$
If $u\in \bE_1(J)\oplus \R_\Gamma$ solves \eqref{P-M eq}, by
$u(0)=u_1(0)+u_2= w_0+H$, we immediately infer that 
$u_2=H$. Now the uniqueness of the solution to \eqref{P-M
  eq-nonlinear} implies $u_1=\tilde{u}$. The uniqueness of
the solution to \eqref{P-M eq} in $\bE_1(J)\oplus \R_\Gamma$ follows.

We are now ready to state the following well-posedness theorem for
\eqref{P-M eq}. 
\begin{theorem}\label{S3.1: PM-thm-ID}
Assume that one of the following conditions holds
\begin{itemize}
\item[] $\varepsilon\in (0,\frac{1}{2})\cup (\frac{1}{2},1)$, $\vartheta\leq -2$ and
  $\displaystyle p>\max\{ \frac{3}{\varepsilon} ,
  -\frac{1+2\varepsilon}{\vartheta}\}$, or 
\item[] $\varepsilon\in (1-\frac{1}{2p},1)$,
  $\vartheta=-2\varepsilon$, and $\displaystyle p>\max\{
  \frac{4}{\varepsilon} , \frac{9}{2}\}.$ 
\end{itemize}
Suppose that $\Qq=[-1,1)$ and that $H$ is a piecewise constant
function on $\Qq$. Let $\Gamma$ be the discontinuity set of $H$. Then,
given any $u_0=H+w_0 $ with 
$$
 w_0\in W^{2-2/p,\vartheta+\frac{2\varepsilon}{p}}_{p,\pi}(\QqG),
$$
equation~\eqref{P-M eq} has a unique solution 
$$
 u \in  L_p(J, W^{2,\vartheta}_{p,\pi}(\QqG)) \cap H^1_p(J,
 L^{\vartheta+2\varepsilon}_{p,\pi}(\QqG))\oplus \R_\Gamma .
$$
for some $J:=[0,T]$ with $T=T(u_0)>0$. Moreover,
$$
 u\in C(J,W^{2-2/p,\vartheta+\frac{2\varepsilon}{p}}_{p,\pi}(\QqG))
 \oplus\R_\Gamma.
$$
\end{theorem} 

\subsection{Two Dimensional Case}

In dimension two, the fractional gradient is defined again via Fourier
series. For a periodic function $u$, $\nabla^{1-\varepsilon} u$ is
defined as 
$$
 \nabla^{1-\varepsilon}u:=\mathcal{F}^{-1}
 \operatorname{diag}\{|k|^{-\varepsilon}\}\mathcal{F}\, |\nabla u|.
$$
The choice of $|\nabla u|$ instead of $\nabla u$ is mainly for
computational simplification.
\begin{lem}\label{S3.2: fractional grad kernel}
For all $u\in C^1_\pi(\Qqq)$ 	
$$
 \nabla^{1-\varepsilon}u(x)= c_\varepsilon \int_{ \Qqq}  \frac{|\nabla
   u|(y)}{|x-y|^{2-\varepsilon}}\, dy +\int_{\Qqq}
 h_\varepsilon(x-y)|\nabla u|(y)\, dy
$$
for some constant $c_\varepsilon>0$ and $h_\varepsilon\in C^\infty$.
\end{lem}
\begin{proof}
It follows from a proof similar to that of Lemma~\ref{S3.1: Sing ker of frac
  grad-ID} and the two dimensional Poisson's summation formula. 
\end{proof}
\begin{lem}\label{S3.2: frac grad-2D}
$$
\nabla^{1-\varepsilon}\in \L(E_{\frac{1}{p}}, BC^1_\pi(\Qqq)).
$$
\end{lem}
\begin{proof}
The proof is the same as that of Lemma~\ref{S3.1: fractional grad-1D}
\end{proof}

We are interested in initial data close to linear combinations of
characteristic functions of  disjoint bounded $C^3$-domains. Just like
in the one dimensional case,  we take a generic initial value function
$H=\chi_{\Omega}$, where $\Omega\subset \mathring{\mathsf{Q}}^2$ is a
bounded $C^3$-domain, and let $\Gamma=\partial\Omega$.

Since $\Omega$ is a set of finite perimeter, it is reasonable to take
$|\nabla H|=\|\partial\Omega\| $. It is known that
$\|\partial\Omega\|=\mathcal{H}^1 \llcorner \Gamma$; 
see \cite[Section 5.1]{Eva92}. For any $\psi\in C^\infty_c(\Qqq) $,
\begin{align*}
 \langle \psi, \nabla^{1-\varepsilon}H \rangle &= \langle
 \mathcal{F}^{-1}|k|^{-\varepsilon}\mathcal{F}\psi, |
 \nabla H|\rangle\\ 
 &=\int_\Gamma \int_{\Qqq}  (\frac{c_\varepsilon}{|x-y|^{2-\varepsilon}}
 + h_\varepsilon(x-y))\psi(x)\, dx\,d\mathcal{H}^1(y)\\ 
 &= \int_{\Qqq} \psi(x) \int_\Gamma
  (\frac{c_\varepsilon}{|x-y|^{2-\varepsilon}}+
  h_\varepsilon(x-y))\,d\mathcal{H}^1(y)\, dx 
\end{align*}
by Fubini Theorem and Lemma~\ref{S3.2: fractional grad kernel}, and 
$\mathcal{H}^n$ is the $n$-dimensional Hausdorff measure. So we
have 
$$
 \nabla^{1-\varepsilon}H(x) = \int_\Gamma
 \frac{c_\varepsilon}{|x-y|^{2-\varepsilon}}\,d\mathcal{H}^1(y)+\int_\Gamma
 h_\varepsilon(x-y)\, d\mathcal{H}^1(y),\quad x\in \QqqG .
$$
Moreover, by its convolution definition, $\nabla^{1-\varepsilon}H\in
C^\infty_\pi(\QqqG)$.
\begin{prop}\label{S3.2: frac der of H-2D}
$$
 \nabla^{1-\varepsilon} H\in BC^{2,1-\varepsilon}_\pi(\QqqG),
$$
and the following estimates hold in a $\delta$-tubular
neighborhood $\sU$ of $\Gamma$
$$
 \quad \d^{1-\varepsilon}\nabla^{1-\varepsilon} H\sim {\bf 1} ,\quad
 \d^{2-\varepsilon}|\nabla\nabla^{1-\varepsilon} H|\sim {\bf 1},\quad
 \d^{3-\varepsilon}|\Delta\nabla^{1-\varepsilon} H|\sim {\bf 1},
$$
along with
$$
\sg(1-2\varepsilon)\d^{2\varepsilon}\Delta \alpha_\varepsilon(H)\sim
{\bf 1} 
$$
\end{prop}
\begin{proof}
Without loss of generality, we may assume that $\Omega$ is simply
connected. More complicated situation can be treated similarly.

(i) Let $\displaystyle I(x):=\int_\Gamma \frac{1}{|x-y|^{2-\varepsilon}}\,
d\mathcal{H}^1(y) $. To estimate $I(x)$ for those $x$ inside a
$\delta$-tubular neighborhood $\sU$ of $\Gamma$, we
first note that there exists a diffeomorphism 
$$
 \Lambda: \sU\to \Gamma\times (-\delta, \delta):\quad x\mapsto
 (\Pi(x), d_\Gamma(x)),
$$
where $\Pi(x)$ is the metric projection of $x$ onto $\Gamma$ and
$d_\Gamma(x)$ is the signed distance from $x$ to
$\Gamma$. $d_\Gamma(x)<0$ if $x$ is in the interior of $\Gamma$.
$$
 \Lambda^{-1}: \Gamma\times (-\delta,\delta)\to \sU:\quad (\p,s)
 \mapsto \p+s\nu_\Gamma(\p),
$$
where $\nu_\Gamma$ denotes the outer normal of $\Gamma$. $\Lambda$ and
$\Lambda^{-1}$ are $C^2$-continuous. For every $x\in \sU$, we pick a
coordinate chart, $\O_x$, around $\Pi(x)$ and chart maps
$\psi_x,\varphi_x$ such that 
$$ 
 \psi_x: \O_x\to (-1,1)\quad \text{with}\quad \varphi_x=\psi_x^{-1}
 \quad \text{and}\quad \psi_x(\Pi(x))=0.
$$
Moreover, $\varphi^*_x g_N|_\Gamma\sim g_1$, the one dimensional Euclidean metric, uniformly in $x$.

To estimate $I(x)$ for $x\in \sU$, first notice that
$$
 \int_{\O_x} \frac{1}{|x-y|^{2-\varepsilon}}\, d\mathcal{H}^1(y)\sim
 \int_{-1}^1 \frac{1}{(y^2+z^2)^{\frac{2-\varepsilon}{2}}}\, dy,
$$
where $z=d_\Gamma(x)$. The Lipschitz constant in this equivalence is
independent of $x$. Without loss of generality, we assume that $z>0$
and $\delta<1$; then
\begin{align*}
 \int_{-1}^1 \frac{dy}{(y^2+z^2)^{\frac{2-\varepsilon}{2}}}
  &=\frac{1}{z^{1-\varepsilon}} \int_{-1/z}^{1/z}
    \frac{dy}{(1+y^2)^{\frac{2-\varepsilon}{2}}}\\ 
 &=\frac{2}{z^{1-\varepsilon}}( \int_0^1
  \frac{dy}{(1+y^2)^{\frac{2-\varepsilon}{2}}}+ \int_1^{1/z}
  \frac{dy}{(1+y^2)^{\frac{2-\varepsilon}{2}}})\\ 
 &\sim \frac{1}{z^{1-\varepsilon}} 
\end{align*}
for $\delta$ sufficiently small. On the other hand, by choosing $\delta$
possibly even smaller, we can always make
$$
 \int_{\Gamma\setminus \O_x}\frac{1}{|x-y|^{2-\varepsilon}}\,
 d\mathcal{H}^1(y)<\frac{1}{2}\int_{\O_x}
 \frac{1}{|x-y|^{2-\varepsilon}}\, d\mathcal{H}^1(y) .
$$
(ii) To estimate $|\nabla I(x)|$, we first compute
$$
 \nabla I(x)= (\varepsilon-2) \int_\Gamma
 \frac{x-y}{|x-y|^{4-\varepsilon}}\, d\mathcal{H}^1(y) .
$$
By the above estimates, it is not hard to see that, in order to bound
$$
 \big |\int_{\O_x} \frac{x-y}{|x-y|^{4-\varepsilon}}\,
 d\mathcal{H}^1(y)\big |,
$$
it suffices to look at 
$$
 \int_{-1}^1 \frac{z\, dy}{(y^2+z^2)^{\frac{4-\varepsilon}{2}}}
 \text{ and }\int_{-1}^1 \frac{|y|\,
   dy}{(y^2+z^2)^{\frac{4-\varepsilon}{2}}}. 
$$
A similar computation as above yields
$$
 \int_{-1}^1 \frac{z\, dy}{(y^2+z^2)^{\frac{4-\varepsilon}{2}}},\:
 \int_{-1}^1 \frac{|y|\, dy}{(y^2+z^2)^{\frac{4-\varepsilon}{2}}}\sim
 \frac{1}{z^{2-\varepsilon}} .
$$
Again by choosing $\delta$ small enough, we can always make 
$$
 \big |\int_{\Gamma\setminus\O_x} \frac{x-y}{|x-y|^{4-\varepsilon}}\,
 d\mathcal{H}^1(y)\big |<\frac{1}{2}\big |\int_{\O_x}
 \frac{x-y}{|x-y|^{4-\varepsilon}}\, d\mathcal{H}^1(y)\big |.
$$
(iii) Since
$$
 \Delta \nabla^{1-\varepsilon}H(x)=(\varepsilon-2)^2 \int_\Gamma
 \frac{1}{|x-y|^{4-\varepsilon}}\, d\mathcal{H}^1(y) +\tilde{h}_\varepsilon(x), 
$$
where $\tilde{h}_\varepsilon\in C^\infty$,
the estimate for $\Delta \nabla^{1-\varepsilon}H$ follows in an
analogous way. Combining everything together, it is clear that
\begin{multline*}
 \d^{1-\varepsilon}(x)\nabla^{1-\varepsilon} H(x)\sim {\bf 1}
 ,\\\d^{2-\varepsilon}(x)|\nabla\nabla^{1-\varepsilon} H(x)|\sim
 {\bf 1},\\\d^{3-\varepsilon}(x)|\Delta\nabla^{1-\varepsilon}
 H(x)|\sim {\bf 1}
\end{multline*}
hold for all $x\in \sU$. The fact that $\nabla^{1-\varepsilon} H\in
BC^{2,1-\varepsilon}_\pi(\QqqG)$ follows from these estimates and the
definition of weighted $BC^k$-spaces.

(iv) As in Lemma~\ref{S3.1: 2nd der of alpha}, direct computations
show that
\begin{multline*}
\Delta\alpha _\varepsilon(H)=
   \alpha
  _\varepsilon(H)^3
  |\nabla^{1-\varepsilon} H|^2
  \big [6|\nabla \nabla^{1-\varepsilon} H|^2    
  -2\nabla^{1-\varepsilon}H \Delta\nabla^{1-\varepsilon}H 
  \big] \\  
  -2\alpha
  _\varepsilon(H)^3\big [
  \nabla^{1-\varepsilon}H \Delta\nabla^{1-\varepsilon}H 
  +|\nabla \nabla^{1-\varepsilon} H|^2  \big] 
  .\qquad
\end{multline*}
Again as in Lemma~\ref{S3.1: 2nd der of alpha}, we only need to
estimate
\begin{align*}
&6|\nabla \nabla^{1-\varepsilon} H|^2    
  -2\nabla^{1-\varepsilon}H \Delta\nabla^{1-\varepsilon}H\\
   \sim &  \Big[ 3\big |\int_\Gamma
 \frac{x-y}{|x-y|^{4-\varepsilon}}\, d\mathcal{H}^1(y)\big |^2
 -\int_\Gamma
 \frac{1}{|x-y|^{4-\varepsilon}}\, d\mathcal{H}^1(y)\int_\Gamma
 \frac{1}{|x-y|^{2-\varepsilon}}\, d\mathcal{H}^1(y) \Big]
\end{align*}
To  estimate the  right hand side, as in (i)-(iii), it suffices to
look at $x\in \sU$ and $y\in \O_x$. 
We need a more precise estimate than those in (i)-(iii), i.e.
 \begin{multline*}
 3\,\big |\int_{\O_x}
 \frac{x-y}{|x-y|^{4-\varepsilon}}\, d\mathcal{H}^1(y)\big |^2
 -\int_{\O_x} \frac{1}{|x-y|^{4-\varepsilon}}\, d\mathcal{H}^1(y)
 \int_{\O_x} \frac{1}{|x-y|^{2-\varepsilon}}\, d\mathcal{H}^1(y) \\
 = 3\Bigl(\int_{-s}^s \frac{z\, dy}{(y^2+z^2)^{\frac{4-\varepsilon}{2}}}J(y)\Bigr)^2 
 + 3\Bigl(\int_{-s}^s \frac{y\, dy}{(y^2+z^2)^{\frac{4-\varepsilon}{2}}}J(y)\Bigr)^2 \\
 -\int_{-s}^s \frac{1}{(y^2+z^2)^{\frac{2-\varepsilon}{2}}}J(y)\, dy  
 \int_{-s}^s \frac{1}{(y^2+z^2)^{\frac{4-\varepsilon}{2}}}J(y)\, dy ,
 \end{multline*} 
where $J(y)\in (K(1-\mu), K(1+\mu))$ for some $K,\mu>0$. $\mu$ is independent of
$x$ and can be chosen arbitrarily small by first shrinking $\O_x$,
or equivalently $s$, and then  $\sU$. Therefore, for each $\mu_0$, by
shrinking $\O_x$ and $\sU$, we have that
 \begin{align*}
&3\big |\int_{\O_x}
 \frac{x-y}{|x-y|^{4-\varepsilon}}\, d\mathcal{H}^1(y)\big |^2
 -\int_{\O_x} \frac{1}{|x-y|^{4-\varepsilon}}\, d\mathcal{H}^1(y)
 \int_{\O_x} \frac{1}{|x-y|^{2-\varepsilon}}\, d\mathcal{H}^1(y) \\
 \leq & 2K_0 \Big\{ (3+\mu_0)\bigl(\int_0^s \frac{z\,
        dy}{(y^2+z^2)^{\frac{4-\varepsilon}{2}}}\bigr)^2
+ 2\mu(3+\mu_0)\bigl(\int_0^s \frac{y\,
        dy}{(y^2+z^2)^{\frac{4-\varepsilon}{2}}}\bigr)^2 \\ 
 &-\int_0^s \frac{dy}{(y^2+z^2)^{\frac{2-\varepsilon}{2}}}
 \int_0^s \frac{dy}{(y^2+z^2)^{\frac{4-\varepsilon}{2}}}
 \Big\}\\
 =&  \frac{2K_0}{z^{4-2\varepsilon}} \Big\{
    (3+\mu_0)\bigl(\int\limits_0^{\frac{s}{z}}
    \frac{dy}{(y^2+1)^{\frac{4-\varepsilon}{2}}}\bigr)^2 
 +2\mu(3+\mu_0)\bigl(\int\limits_0^{\frac{s}{z}}
    \frac{y\, dy}{(y^2+1)^{\frac{4-\varepsilon}{2}}}\bigr)^2 \\ 
 &-\int\limits_0^{\frac{s}{z}} \frac{dy}{(y^2+1)^{\frac{2-\varepsilon}{2}}}
 \int\limits_0^{\frac{s}{z}} \frac{dy}{(y^2+1)^{\frac{4-\varepsilon}{2}}}
 \Big\}\\
 \end{align*}
 and 
\begin{align*}
&3\big |\int_{\O_x}
 \frac{x-y}{|x-y|^{4-\varepsilon}}\, d\mathcal{H}^1(y)\big |^2
 -\int_{\O_x} \frac{1}{|x-y|^{4-\varepsilon}}\, d\mathcal{H}^1(y)
 \int_{\O_x} \frac{1}{|x-y|^{2-\varepsilon}}\, d\mathcal{H}^1(y) \\
 \geq & K_0 \Big\{ (3-\mu_0)\bigl(\int_{-s}^s
        \frac{z\, dy}{(y^2+z^2)^{\frac{4-\varepsilon}{2}}}\bigr)^2 
 -\int_{-s}^s \frac{dy}{(y^2+z^2)^{\frac{2-\varepsilon}{2}}}
 \int_{-s}^s \frac{dy}{(y^2+z^2)^{\frac{4-\varepsilon}{2}}}
 \Big\}\\
 =&  \frac{2K_0}{z^{4-2\varepsilon}} \Big\{
    (3-\mu_0)\bigl(\int\limits_0^{\frac{s}{z}}
    \frac{dy}{(y^2+1)^{\frac{4-\varepsilon}{2}}}\bigr)^2
 -\int\limits_0^{\frac{s}{z}} \frac{dy}{(y^2+1)^{\frac{2-\varepsilon}{2}}}
 \int\limits_0^{\frac{s}{z}} \frac{dy}{(y^2+1)^{\frac{4-\varepsilon}{2}}}
 \Big\}\\
 \end{align*} 
 for some $K_0>0$. 
Recall that $\mu$, and thus $2\mu(3+\mu_0)$, can be made arbitrarily
small, and note that once $\O_x$, i.e. $s$, is fixed, $s/z$ can be
made arbitrarily large by further shrinking $\sU$. Therefore, we have
 \begin{multline*}
 \d^{4-2\varepsilon} \Big[3\big |\int_{\O_x}
 \frac{x-y}{|x-y|^{4-\varepsilon}}\, d\mathcal{H}^1(y)\big |^2
 +\\-\int_{\O_x} \frac{1}{|x-y|^{4-\varepsilon}}\, d\mathcal{H}^1(y)
 \int_{\O_x} \frac{1}{|x-y|^{2-\varepsilon}}\, d\mathcal{H}^1(y)\Big]
 \sim {\bf 1}, 
 \end{multline*}
 as long as
 \begin{equation}
 \label{S3.2: Cond}
 3\int\limits_0^\infty \frac{dy}{(y^2+1)^{\frac{4-\varepsilon}{2}}}
 -\int\limits_0^\infty \frac{dy}{(y^2+1)^{\frac{2-\varepsilon}{2}}}
 \neq 0.
 \end{equation}
One verifies that
$$
3\int\limits_0^\infty \frac{dy}{(y^2+1)^{\frac{4-\varepsilon}{2}}}
 -\int\limits_0^\infty \frac{dy}{(y^2+1)^{\frac{2-\varepsilon}{2}}}
 =\frac{3}{2} B(\frac{1}{2}, \frac{3-\varepsilon}{2})  -\frac{1}{2}B(\frac{1}{2}, \frac{1-\varepsilon}{2}),
$$
where $B(p,q)=\int_0^1 x^{p-1} (1-x)^{q-1}\, dx$ is the Beta function. The right hand side equals zero iff
$\varepsilon = 1/2$.
Thus, we conclude that for all $\varepsilon\neq 1/2$ 
 $$
  \sg(1-2\varepsilon)\d^{2\varepsilon}\Delta \alpha_\varepsilon(H) \sim {\bf 1}
 $$
 in a sufficiently small $\delta$-tubular neighborhood $\sU$ of $\Gamma$.
\end{proof}

Recall that $\R_\Gamma$ denotes the set of all functions that are constants in
each connected component of $\QqqG$. Now, combining Lemma~\ref{S3.2:
  frac grad-2D}, Proposition~\ref{S3.2: frac der of H-2D}, and an
argument analogous to the one used in the one dimensional case, we
obtain the following proposition.
\begin{prop}
Let $1<p<\infty$ and $\varepsilon$ satisfy \eqref{S3: ASP: vartheta & p} or 
\eqref{S3: ASP: vartheta & p-2}. Then, for each $w\in E_{\frac{1}{p}}$, the operator
$$
\sA(w)\in \mathcal{MR}_p(W^{2,\vartheta}_{p,\pi}(\QqqG),L^{\vartheta+2\varepsilon}_{p,\pi}(\QqqG) ).
$$
\end{prop}
\begin{proof}
As in the proof for Proposition~\ref{S3.1: MR-thm}, we put $\displaystyle h=\sg(1-2\varepsilon)\log \alpha_\varepsilon(H)$.
Easy computations show
\begin{align*}
|\alpha_\varepsilon(H)^{\frac{1}{2-2\varepsilon}}\nabla h|
\sim  |\d^{2\varepsilon-1} \frac{d}{dx} \alpha_\varepsilon(H)|\sim {\bf 1},
\end{align*}
and 
\begin{align*}
\alpha_\varepsilon(H)^{\frac{2\varepsilon}{2-2\varepsilon}}
  \div (\alpha_\varepsilon(H) \nabla h)\sim
   \sg(1-2\varepsilon)\d^{2\varepsilon} \Delta
      \alpha_\varepsilon(H) \sim {\bf 1}. 
\end{align*}
near $\Gamma$. Then the rest of the proof follows in the same way as that for Proposition~\ref{S3.1: MR-thm}.
\end{proof}
The following theorem concerning the local
wellposedness of equation \eqref{P-M eq} in two space dimensions follows.
\begin{theorem}\label{S3.2: PM-thm-2D}
Assume that one of the following conditions holds
\begin{itemize}
\item[] $\varepsilon\in (0,\frac{1}{2})\cup (\frac{1}{2},1)$, $\vartheta\leq -2$ and
  $\displaystyle p>\max\{ \frac{4}{\varepsilon} ,
  -\frac{2+2\varepsilon}{\vartheta}\}$ or
\item[] $\varepsilon\in (1-\frac{1}{2p}, 1)$,
  $\vartheta=-2\varepsilon$, and $\displaystyle p>\max\{
  \frac{6}{\varepsilon}, \frac{13}{2}\}.$ 
\end{itemize}
Suppose that $H$ is a linear combination of characteristic functions of
disjoint $C^3$-domains $\Omega_i$ in $\mathring{\mathsf{Q}}^2$. 
Let $\Gamma=\cup_i \partial \Omega_i $. Given any $u_0=H+w_0 $ with 
$$
 w_0\in W^{2-2/p,\vartheta+\frac{2\varepsilon}{p}}_{p,\pi}(\QqqG),
$$
equation~\eqref{P-M eq} has a unique solution
$$
 u \in  L_p(J, W^{2,\vartheta}_{p,\pi}(\QqqG)) \cap H^1_p(J,
 L^{\vartheta+2\varepsilon}_{p,\pi}(\QqqG))\oplus \R_\Gamma .
$$
for some $J:=[0,T]$ with $T=T(u_0)>0$. Moreover,
$$
u\in C(J,
W^{2-2/p,\vartheta+\frac{2\varepsilon}{p}}_{p,\pi}(\QqqG))\oplus
\R_\Gamma.
$$
\end{theorem} 

\section{\bf Global existence}

In this section, we focus on  the case \eqref{S3: ASP: vartheta & p-2}
\begin{equation*}
\vartheta=-2\varepsilon,\: p>\max\{ \frac{2N+2}{\varepsilon},
\frac{4N+5}{2}\},\: \varepsilon>1-\frac{1}{2p}
\end{equation*}
and prove global existence of the solutions to \eqref{P-M eq} to
initial data close enough to an equilibrium. 
Note that \eqref{S3: ASP: vartheta & p-2} implies the necessary 
condition $\varepsilon>1/2$ in the sequel, 
and this is why only \eqref{S3: ASP: vartheta & p-2} 
is considered in this section.

In \cite[Proposition~6]{Gui09}, the first author proves that
characteristic functions of smooth domains $\Omega$ are stationary
solutions for \eqref{P-M eq}. While in that article, the
submanifold $\Gamma=\partial\Omega$ is required to be smooth, lower
regularity, e.g. $C^3$-regularity, suffices.
\begin{prop}
Linear combinations of characteristic functions of disjoint
$C^3$-domains $\Omega_i$ in $\mathring{\sf Q}^N$ are stationary
solutions to \eqref{P-M eq}. 
\end{prop}
We define 
$$
P:  W^{2,-2\varepsilon}_{p,\pi}(\QG)\to L_{p,\pi}(\QG): \quad 
 u\mapsto \div \bigl(\frac{1}{1+|\nabla^{1-\varepsilon}(H+u)|^2}\nabla
 u\bigr),
$$
where $H=\chi_\Omega$ for some $C^3$-domain $\Omega\subset
\mathring{\mathsf{Q}}^N$. 
The discussions in the previous section (cf. \eqref{S3.1: Reg-main}) show that
$$
P\in C^\omega(W^{2,-2\varepsilon}_{p,\pi}(\QG), L_{p,\pi}(\QG)).
$$ 
Let
$$
 \sA_{\alpha_\varepsilon} u:=\div(\alpha _\varepsilon\nabla u),\:
 \alpha _\varepsilon:=\frac{1}{1+|\nabla^{1-\varepsilon}H|^2}.
$$
Note that $\alpha _\varepsilon\sim \d^{2-2\varepsilon}$. Denote the
Fr\'echet derivative of $P$ at $0$ by $\partial P(0)$. Then an easy
computation shows that $\partial P(0)=\sA_{\alpha_\varepsilon}$.
Consider the following abstract linear equation.
\begin{equation}\label{S4: ALE}
\begin{cases}
 \partial_t u -\sA_{\alpha _\varepsilon} u=0  &\text{in
 }\Q\times(0,\infty),\\  
 u&\text{periodic,}\\
 u(0)=u_0  &\text{in }\Q .
\end{cases}
\end{equation}
We can associate with $\sA_{\alpha _\varepsilon}$ a form operator $\a$
with $D(\a)=\mathring{H}^1_{\alpha_\varepsilon,\pi}(\QG)$, defined by
\begin{align*}
\a(u,v)=\int_{\Q} \alpha _\varepsilon (\nabla u | \nabla v )\, dx,
\end{align*}
for $u,v\in D(\a)$. Here $\mathring{H}^1_{\alpha_\varepsilon,\pi}(\QG)$ is the closure
of $\mathcal{D}_\pi(\QG)$, where $\mathcal{D}_\pi(\QG)= \phi^*\mathcal{D}(\M)$, with respect to the
norm $\|\cdot\|_{\alpha_\varepsilon}$, 
$$ 
\|u\|_{\alpha_\varepsilon}=(\|u\|_2^2 +\|\sqrt{\alpha_\varepsilon}\nabla u\|_2^2)^{1/2}
$$
with $\|\cdot\|_2$ being the norm of $L_{2,\pi}(\QG)$.
\begin{lem}\label{S4: cpt embed Trace of H1 Poincare ineq}
\begin{itemize}
   \item[]
  \item[(i)] The embedding $D(\a)\hookrightarrow L_{2,\pi}(\QG)$ is compact. 
  \item[(ii)] Any function $u\in D(\a)$ admits a trace
    $\gamma_{\Gamma}(u)=0$ a.e. on $\Gamma$. 
  \item[(iii)] It holds that
$$
 \|u\|_2\leq C\| \sqrt{\alpha _\varepsilon}\,\nabla u\|_2,\:
 u\in D(\a),
$$
where $\|\cdot\|_2$ is the norm of $L_{2,\pi}(\QG)$. 
\end{itemize}
\end{lem}
\begin{proof}
(i) Since $\alpha _\varepsilon\sim \d^{2-2\varepsilon}$ and
$\varepsilon>1/2$, there exists an $q>1$ such that 
$$\int_{\QG} \frac{1}{\alpha _\varepsilon^q(x)}\, dx<\infty. $$
Then one has that $|\nabla u|\in W^1_{1+s,\pi}(\QG)$ for some small
enough $s>0$ since
\begin{multline}\label{S4: 1+s space}
 \int_{\QG} |\nabla u|^{1+s}\, dx \leq \int_{\QG}
 (\frac{\sqrt{\alpha_\varepsilon(x)}}{\sqrt{\alpha
  _\varepsilon(x)}})^{1+s}|\nabla u|^{1+s}\, dx\\ 
  \leq (\int_{\QG} \alpha _\varepsilon(x)^{-\frac{1+s}{1-s}}\,
  dx)^{\frac{1-s}{2}} (\int_{\QG} \alpha _\varepsilon(x) |\nabla
  u(x)|^2\, dx)^{\frac{1+s}{2}}<\infty 
\end{multline}
provided that $\frac{1+s}{1-s}<q$, which is always possible for a
small enough $s$. This shows that $u\in W^{1}_{1+s,\pi}(\QG)$. The
claim therefore follows from the compactness part of Sobolev embedding
theorem. This is obvious for $N=1$. For $N=2$, it follows observing
that $2<(1+s)^*= \frac{N(1+s)}{N-1-s}$ is valid as long as
$N<2\frac{1+s}{1-s}$. The latter is always the case for $N=2$.\\[0.1cm]
(ii) Inequality \eqref{S4: 1+s space} implies that on each
component $\Omega_i$ of $\QG$, 
$$
 D(\a)\hookrightarrow
 W^1_{1+s}(\Omega_i)
$$
for some $s>0$ small.
By the well known trace theorem,
$$
 \gamma_{\Gamma}\in \L\bigl(W^1_{1+s}(\Omega_i),
 W^{1-\frac{1}{1+s}}_{1+s}(\partial{\Omega}_i)\bigr).
$$
Therefore the trace operator is well-defined on $D(\a)$ and 
\begin{equation}\label{S4: cont of trace op}
\gamma_{\Gamma}\in \L\bigl(D(\a),
W^{1-\frac{1}{1+s}}_{1+s}(\partial{\Omega}_i)\bigr)
\end{equation} 
on each connected component of $\QG$.
By the density of $\mathcal{D}_\pi(\QG)$ in  $D(\a)$,
we can take a sequence $(u_k)_{k\in\N}\subset \mathcal{D}_\pi(\QG) $
converging to $u$ in $D(\a)$. Since $\gamma_{\Gamma}(u_k)=0$, we
conclude from \eqref{S4: cont of trace op} that 
$\gamma_{\Gamma}(u)=0$ as well.\\[0.1cm]
(iii) Given any $u\in D(\a)$, it follows from \eqref{S4: 1+s
    space} that $u\in W^1_{1+s,\pi}(\QG)$ with $s$ small enough, and
  by the trace lemma, we have $\gamma_\Gamma(u)=0$.
So we can apply the Poincar\'e inequality for $W^1_{1+s,\pi}(\QG)$ on
each connected component of $\QG$ to $u$, which yields 
$$
 \|u\|_{L_{1+s,\pi}(\QG)} \leq C\|\nabla u\|_{L_{1+s,\pi}(\QG)}.
$$
In view of the embedding $W^1_{1+s,\pi}(\QG) \hookrightarrow
L_{2,\pi}(\QG)$ and \eqref{S4: 1+s space}, it holds that
$$
 \|u\|_2 \leq C\|u \|_{W^1_{1+s,\pi}(\QG)}\leq C\|\nabla
 u\|_{L_{1+s,\pi}(\QG)}\leq C\| \sqrt{\alpha _\varepsilon}\,\nabla
 u\|_2.
$$
\end{proof}
\begin{prop}\label{S4: cont-L_2-coer}
$\a$ is continuous and $D(\a)$-coercive. More precisely,
\begin{itemize}
\item[(i)](Continuity) there exists some constant $C$ such that for
all $u,v\in D(\a)$ 
$$
 |\a(u,v)|\leq C\|u\|_{D(\a)} \|v\|_{D(\a)}.
$$
\item[(ii)]($D(\a)$-Coercivity) There is some $C$ such that for any
  $u\in D(\a)$ 
$$
 \Rp(\a(u,u)) \geq C \|u\|^2_{D(\a)}.
$$
\end{itemize}
\end{prop}
\begin{proof}
(i)
\begin{align*}
|\a(u,v)|&=|\int_{\Q} \alpha _\varepsilon (\nabla u | \nabla v )\, dx |\\
&\leq\int_{\Q} \alpha _\varepsilon\, \d^{2\varepsilon-2} |\d^{2-2\varepsilon}(\nabla u | \nabla v )|\, dx\\
&\leq C\int_{\Q} |(\d^{1-1\varepsilon}\nabla u | \d^{1-1\varepsilon}\nabla v )|\, dx\\
&\leq C \|u\|_{D(\a)} \|v\|_{D(\a)}.
\end{align*}
The last step follows from H\"older inequality and $|(a|b)|\leq |a| |b|$.\\
(ii) It is a direct consequence of Lemma~\ref{S4: cpt embed Trace of
  H1 Poincare ineq}(iii) that 
\begin{align*}
\Rp(\a(u,u))&=\a(u,v)=\int_{\Q} \alpha _\varepsilon |\nabla u|^2\, dx
              \geq C\|u\|_{D(\a)}. 
\end{align*}
\end{proof}

Proposition \ref{S4: cont-L_2-coer} shows that $\a$ with $D(\a)$ is
densely defined, sectorial and closed on $L_{2,\pi}(\QG)$. By
\cite[Theorems~VI.2.1, IX.1.24]{Kato80}, we can find an associated
operator $T$ such that $-T$ generates a strongly continuous analytic
semigroup of contractions on $L_{2,\pi}(\QG)$, i.e. satisfying
$\|e^{-tT}\|_{\L(L_{2,\pi}(\QG))}\leq 1$ for all $t\geq 0$. Its domain
is given by  
$$
 D(T):=\big\{u\in D(\a): \exists !\: v\in L_{2,\pi}(\QG)\text{ s.t. }\a(u,\phi)=\langle
 v , \phi \rangle, \forall \phi\in D(\a)\big\}
$$ 
and $Tu=v$; $D(T)$ is a core of $\a$. The operator $T$ is unique in
the sense that there exists only one operator satisfying
$$
\a(u,v)= \langle T u, v \rangle,\quad u\in D(T),\, v\in D(\a).
$$
We have proved that
$$
W^{2,-2\varepsilon}_{2,\pi}(\QG) \overset{d}{\hookrightarrow}
D(\a)\overset{d}{\hookrightarrow} \zW^1_{1+s,\pi}(\QG)
\hookrightarrow
W^1_{1+s,\pi}(\QG)
.$$
Here $\zW^1_{1+s,\pi}(\QG)$ is the closure of $\mathcal{D}_\pi(\QG)$ in 
$W^1_{1+s,\pi}(\QG)$.
Then  we can uniquely extend $\sA_{\alpha_\varepsilon}$, which is originally defined
on $W^{2,-2\varepsilon}_{2,\pi}(\QG)$ as in Section~3,
to $\zW^1_{1+s,\pi}(\QG)$. 
Now $\sA_{\alpha_\varepsilon}$ can be defined on $\zW^1_{1+s,\pi}(\QG)$ by
$$
\langle \sA_{\alpha_\varepsilon} u ,v\rangle= -\a(u,v)
,\quad
u\in \zW^1_{1+s,\pi}(\QG), v\in \mathcal{D}_\pi(\QG)
$$
and 
$\sA_{\alpha_\varepsilon} \in \L( \zW^1_{1+s,\pi}(\QG), (\zW^1_{1+s,\pi}(\QG))^\prime) $.
Restricted onto $D(\a)$ and by a density argument, this yields
that for any $u,v\in D(\a)$ 
$$\langle \sA_{\alpha_\varepsilon} u ,v\rangle= -\a(u,v),  $$
and thus
$$
|\langle \sA_{\alpha_\varepsilon}u, v\rangle|\leq  C\|u\|_{D(\a)}\|v\|_{D(\a)}, 
$$
which implies that $\sA_{\alpha_\varepsilon}\in
\L(D(\a),
(D(\a))^\prime)$. 
Since it holds that
$$
\sA_{\alpha_\varepsilon}\in \L( W^{2,-2\varepsilon}_{2,\pi}(\QG), L_{2,\pi}(\QG))
$$ 
supported by Lemmas~\ref{S2: nabla} and \ref{S2: pointwise mul}, we further
have that, for any $u\in  W^{2,-2\varepsilon}_{2,\pi}(\QG)$ and $v\in D(\a)$,
$$
|\a(u,v)|=|\langle \sA_{\alpha_\varepsilon}u, v\rangle| 
\leq  \|\sA_{\alpha_\varepsilon}u\|_2 \|v\|_2 
\leq C\| u\|_{2,2;-2\varepsilon}\|v\|_2.
$$ 
It is known that a function $u\in D(T)$ iff $u\in D(\a)$ and 
$$
|\a(u,v)|\leq C\|v\|_2,\quad v\in D(\a).
$$
Therefore, we conclude that
$$ 
T= \sA_{\alpha_\varepsilon}|_{D(T)}
\quad \text{and}\quad
W^{2,-2\varepsilon}_{2,\pi}(\QG)\subset D(T).
$$
On the other hand, by picking $w=0$ in Proposition~\ref{S3.1: MR-thm}
yields 
$$ 
\sA_{\alpha_\varepsilon}\in
\mathcal{MR}_p(W^{2,-2\varepsilon}_{2,\pi}(\QG),L_{2,\pi}(\QG)),\:
1<p<\infty.
$$
It is well known, see e.g. \cite[Proposition~1.2]{Pru03}, that this
implies the existence of some $\omega\geq 0$ such that 
$$
\omega+\sA_{\alpha_\varepsilon}\in
\Lis(W^{2,-2\varepsilon}_{2,\pi}(\QG),L_{2,\pi}(\QG))\cap
\S(L_{2,\pi}(\QG))
$$
with spectral angle $\phi_{\omega+\sA_{\alpha_\varepsilon}} <\pi/2$.

Due to well-known results of semigroup theory, we know that for the
same $\omega$ as above
$$
\omega+\sA_{\alpha_\varepsilon}\in
\Lis(D(T),L_{2,\pi}(\QG)),
$$
from which we infer right away that 
$$
D(T)\doteq W^{2,-2\varepsilon}_{2,\pi}(\QG)
.$$
By standard real analysis knowledge, we know that $u\in D(\a)$
implies the validity of $(|u|-1)^+  \sg u \in D(\a)$ and that
\begin{align*}
\nabla\bigl[(|u|-1)^+  \sg u\bigr]=
\begin{cases}
\nabla u, \quad & |u|>1;\\
0, & |u|\leq 1.
\end{cases}
\end{align*}
Here it is understood that 
\begin{align*}
\sg u:=
\begin{cases}
u/|u|, \quad & u\neq 0;\\
0, &u=0.
\end{cases}
\end{align*}
Now it is clear that
$$\Rp\bigl[\a(u,(|u|-1)^+  \sg u)\bigr]\geq 0 .$$
By \cite[Theorem~2.7]{Ouh92}, the semigroup
$\{e^{-t\sA_{\alpha_\varepsilon}}\}_{t\geq 0}$ is
$L_\infty$-contractive, or more precisely, 
$$
\|e^{-t\sA_{\alpha_\varepsilon}} u\|_\infty \leq \|u\|_\infty,\quad
t\geq 0,\quad   u\in L_{2,\pi}(\QG)\cap L_{\infty,\pi}(\QG). 
$$
We can then follow a well-known argument, see
\cite[Chapter~1.4]{Dav89}, to prove that for each $1<p<\infty$,
$\{e^{-t\sA_{\alpha_\varepsilon}}\}_{t\geq 0}$ can be extended to a
strongly continuous analytic semigroup of contractions on
$L_{p,\pi}(\QG)$. Then we can determine the domain for this semigroup
by the same argument used previously for the semigroup on
$L_{2,\pi}(\QG)$. 
In sum, we can prove the following assertion. 
\begin{lem}\label{S4: semigroup}
$-\sA_{\alpha_\varepsilon}$ generates a strongly
continuous analytic semigroup of contractions on $L_{p,\pi}(\QG)$ with domain
$W^{2,-2\varepsilon}_{p,\pi}(\QG)$ for all $1<p<\infty$.
\end{lem}

Now we apply the form operator method to the operator
$\sA_{\alpha_\varepsilon} -\omega$ for some sufficiently small
positive $\omega$. By Lemma~\ref{S4: cpt embed Trace of H1 Poincare
  ineq}(iii), we infer that Proposition~\ref{S4: cont-L_2-coer} still
holds true for $\sA_{\alpha_\varepsilon} -\omega$ with $\omega$ small.
Then we can follow the above argument step by step and prove the same
contraction semigroup property for $\sA_{\alpha_\varepsilon} -\omega$
as in Lemma~\ref{S4: semigroup}. This  immediately gives a spectral
bound for $\sA_{\alpha_\varepsilon}$.
\begin{lem}\label{S4: spectrum bd}
$\sup\big\{ \Rp(\mu): \, \mu\in \sigma(-\sA_{\alpha_\varepsilon})\big\}<0.$
\end{lem}

The (exponential) asymptotic stability of the stationary solution $H$
now follows from well-known linearized stability results.

\begin{theorem}\label{S4: Global existence}
Assume that
$\Q=[-1,1)^N$ with $N=1,2$ and 
$$
\varepsilon\in (1-\frac{1}{2p},1),\quad p>\max\{ \frac{2N+2}{\varepsilon},
\frac{4N+5}{2}\}. 
$$
Suppose that $\Gamma$ is a $C^3$-submanifold in
$\mathring{\mathsf{Q}}^N$.
Let $H$ be a component-wise constant function on $\QG$. Then $H$ is a
stationary solution to \eqref{P-M eq} and attracts all solutions which
are initially
$W^{2-\frac{2}{p},\frac{2\varepsilon(1-p)}{p}}_{p,\pi}(\QG)$ close to
$H$. 

More precisely, if the initial datum satisfies
$$
 u_0=H+w_0\text{ with }w_0\in
 W^{2-\frac{2}{p},\frac{2\varepsilon(1-p)}{p}}_{p,\pi}(\QG)
$$ 
and $\|w_0\|_{2-\frac{2}{p},p;\frac{2\varepsilon(1-p)}{p}}$ sufficiently
small, then the solution $u$ to \eqref{P-M eq} converges to $H$
exponentially fast in
$W^{2-\frac{2}{p},\frac{2\varepsilon(1-p)}{p}}_{p,\pi}(\QG)$-topology, 
in particular, in $C^1(\Q)$-topology.
\end{theorem}

\end{document}